\documentclass[12pt]{amsart}
\usepackage{amsfonts,amssymb,amscd,amsmath,enumerate,verbatim}
\usepackage[latin1]{inputenc}
\usepackage{amscd}
\usepackage{latexsym}
\usepackage{mathptmx} 
\usepackage{multicol}
\usepackage{epsfig}
\usepackage{color}
\definecolor{cream}{cmyk}{0,0,0.3,0}
\definecolor{orange}{cmyk}{0,0.61,0.87,0}
\definecolor{whiteblue}{cmyk}{0.2,0,0,0}
\definecolor{whitepeach}{cmyk}{0,0.2,0,0}
\definecolor{peach}{cmyk}{0,0.2,0.2,0}
\definecolor{bluegreen}{cmyk}{0.2,0,0.1,0}
\definecolor{whitepurple}{cmyk}{0.3,0.2,0.1,0}

\input xy
\xyoption{all}

\def\NZQ{\mathbb}               

\def\ZZ{{\NZQ Z}}

\def\frk{\mathfrak}               

\def\Phi{{\frk N}}
\def\opn#1#2{\def#1{\operatorname{#2}}} 
\opn\chara{char} 
\opn\length{\ell} 
\opn\pd{pd} 
\opn\rk{rk}
\opn\projdim{proj\,dim} 
\opn\injdim{inj\,dim} 
\opn\rank{rank}
\opn\depth{depth} 
\opn\grade{grade} 
\opn\height{height}
\opn\embdim{emb\,dim} 
\opn\codim{codim}

\opn\Tr{Tr} 
\opn\bigrank{big\,rank}
\opn\superheight{superheight}
\opn\lcm{lcm}
\opn\trdeg{tr\,deg}
\opn\reg{reg} 
\opn\lreg{lreg} 
\opn\ini{in} 
\opn\lpd{lpd}
\opn\size{size}
\opn\mult{mult}
\opn\dist{dist}
\opn\cone{cone}
\opn\lex{lex}
\opn\rev{rev}
\opn\div{div} \opn\Div{Div} \opn\cl{cl} \opn\Cl{Cl}
\opn\Spec{Spec} \opn\Supp{Supp} \opn\supp{supp} \opn\Sing{Sing}
\opn\Ass{Ass} \opn\Min{Min}
\opn\Ann{Ann} \opn\Rad{Rad} \opn\Soc{Soc}
\opn\Syz{Syz} \opn\Im{Im} \opn\Ker{Ker} \opn\Coker{Coker}
\opn\Am{Am} \opn\Hom{Hom} \opn\Tor{Tor} \opn\Ext{Ext}
\opn\End{End} \opn\Aut{Aut} \opn\id{id} \opn\ini{in}

\opn\nat{nat}
\opn\pff{pf}
\opn\Pf{Pf} \opn\GL{GL} \opn\SL{SL} \opn\mod{mod} \opn\ord{ord}
\opn\Gin{Gin}
\opn\Hilb{Hilb}\opn\adeg{adeg}\opn\std{std}\opn\ip{infpt}
\opn\Pol{Pol}
\opn\sat{sat}
\opn\Var{Var}
\opn\Gen{Gen}

\opn\aff{aff} \opn\con{conv} \opn\relint{relint} \opn\st{st}
\opn\lk{lk} \opn\cn{cn} \opn\core{core} \opn\vol{vol}
\opn\link{link} \opn\star{star}
\opn\gr{gr}

\def\pot#1#2{#1[\kern-0.28ex[#2]\kern-0.28ex]}

\opn\dirlim{\underrightarrow{\lim}}
\opn\inivlim{\underleftarrow{\lim}}
\let\to=\rightarrow

\def\Implies{\ifmmode\Longrightarrow \else
        \unskip${}\Longrightarrow{}$\ignorespaces\fi}
\def\implies{\ifmmode\Rightarrow \else
        \unskip${}\Rightarrow{}$\ignorespaces\fi}
\def\iff{\ifmmode\Longleftrightarrow \else
        \unskip${}\Longleftrightarrow{}$\ignorespaces\fi}

\let\:=\colon
\newtheorem{Theorem}{Theorem}[section]
\newtheorem{Lemma}[Theorem]{Lemma}
\newtheorem{Corollary}[Theorem]{Corollary}
\newtheorem{Proposition}[Theorem]{Proposition}
\newtheorem{Remark}[Theorem]{Remark}

\newtheorem{Example}[Theorem]{Example}

\newtheorem{Question}[Theorem]{Question}

\numberwithin{equation}{section}

\let\epsilon\varepsilon
\let\phi=\varphi
\let\kappa=\varkappa
\def\qed{\ifhmode\textqed\fi
      \ifmmode\ifinner\quad\qedsymbol\else\dispqed\fi\fi}
\def\textqed{\unskip\nobreak\penalty50
       \hskip2em\hbox{}\nobreak\hfil\qedsymbol
       \parfillskip=0pt \finalhyphendemerits=0}
\def\dispqed{\rlap{\qquad\qedsymbol}}

\opn\dis{dis}
\opn\height{height}
\opn\dist{dist}
\def\pnt{{\raise0.5mm\hbox{\large\bf.}}}

\opn\Lex{Lex}

\begin{document}
\title{Regularity and $a$-invariant of Cameron--Walker graphs}
\author{Takayuki Hibi, Kyouko Kimura, Kazunori Matsuda and Akiyoshi Tsuchiya}
\address{Takayuki Hibi,
Department of Pure and Applied Mathematics,
Graduate School of Information Science and Technology,
Osaka University, Suita, Osaka 565-0871, Japan}
\email{hibi@math.sci.osaka-u.ac.jp}

\address{Kyouko Kimura,
Department of Mathematics, 
Faculty of Science, 
Shizuoka University, 836 Ohya, Suruga-ku, Shizuoka 422-8529, Japan}
\email{kimura.kyoko.a@shizuoka.ac.jp}

\address{Kazunori Matsuda,
Kitami Institute of Technology, 
Kitami, Hokkaido 090-8507, Japan}
\email{kaz-matsuda@mail.kitami-it.ac.jp}

\address{Akiyoshi Tsuchiya,
Department of Pure and Applied Mathematics,
Graduate School of Information Science and Technology,
Osaka University, Suita, Osaka 565-0871, Japan}
\email{a-tsuchiya@ist.osaka-u.ac.jp}
\begin{abstract}
Let $S$ be the polynomial ring over a field $K$ and $I \subset S$ a homogeneous ideal. 
Let $h(S/I,\lambda)$ be the $h$-polynomial of $S/I$ and $s = \deg h(S/I,\lambda)$ the degree of $h(S/I,\lambda)$. 
It follows that the inequality $s - r \leq d - e$, 
where $r = \reg(S/I)$, $d = \dim S/I$ and $e = \depth S/I$, 
is satisfied and, in addition, the equality $s - r = d - e$ holds 
if and only if $S/I$ has a unique extremal Betti number. 
We are interested in finding a natural class of finite simple graphs $G$ for which $S/I(G)$, 
where $I(G)$ is the edge ideal of $G$, satisfies $s - r = d - e$. 
Let $a(S/I(G))$ denote the $a$-invariant of $S/I$, 
i.e., $a(S/I(G)) = s - d$.  One has $a(S/I(G)) \leq 0$. 
In the present paper, by showing the fundamental fact that every Cameron--Walker graph $G$ satisfies $a(S/I(G)) = 0$, 
a classification of Cameron--Walker graphs $G$ for which $S/I(G)$ satisfies $s - r = d - e$ will be exhibited.
\end{abstract}
\maketitle
\section*{Introduction}
In the current trends on combinatorial and computational commutative algebra, 
the study on regularity of edge ideals of finite simple graphs becomes 
fashionable and many papers including 
\cite{BC1, DHS, HaVanTuyl, Ku, W} 
have been published.  
In the present paper we are interested in the regularity and the $h$-polynomials of edge ideals.

Let $S = K[x_1, \ldots, x_n]$ denote the polynomial ring in 
$n$ variables over a field $K$ with each $\deg x_i = 1$ and 
$I \subset S$ a homogeneous ideal of $S$ with $\dim S/I = d$.  
The Hilbert series $H\left(S/I, \ \lambda \right)$ of $S/I$ is of the form 
\[
H\left(S/I, \ \lambda \right) = \frac{h_0 + h_1\lambda + h_2\lambda^2 + \cdots + h_s\lambda^s}{(1 - \lambda)^d}, 
\]
where each $h_i \in \ZZ$ (\cite[Proposition 4.4.1]{BH}).  
We say that 
\[
h\left(S/I, \ \lambda \right) = h_0 + h_1\lambda + h_2\lambda^2 + \cdots + h_s\lambda^s
\] 
with $h_s \neq 0$ is the {\em $h$-polynomial} of $S/I$.  
We call the difference 
$\deg h\left(S/I, \ \lambda \right) - \dim S/I$ 
the $a$-invariant (\cite[Definition 4.4.4]{BH}) of $S/I$ and denote it by $a(S/I)$. 
It is known that $a(S/I) \leq 0$ if $I$ is a squarefree monomial ideal. 

\par
Let
\[
{\bf F}_{S/I} : 0 \to \bigoplus_{j \geq 1} S(-(p+j))^{\beta_{p, p + j}(S/I)} \to \cdots \to \bigoplus_{j \geq 1} S(-(1+j))^{\beta_{1, 1 + j}(S/I)} \to S \to S/I \to 0
\]
be the minimal graded free resolution of $S/I$ over $S$, where $p$ is the projective dimension of $S/I$. 
The ({\em Castelnuovo--Mumford}\,) {\em regularity} of $S/I$ is  
\[
\reg\left(S/I\right) = \max\{j \, : \, \beta_{i, i + j} (S/I) \neq 0\}.
\] 
The inequality
\begin{equation}
   \label{eq:srde}
\deg h\left(S/I, \ \lambda \right) -  \reg\left( S/I \right) \leq \dim S/I - \depth\left( S/I \right)
\end{equation}
is well known (\cite[Corollary B.4.1]{V}) and its proof is easy. 
In fact, since \cite[Lemma 4.1.13]{BH} says that 
\[
H\left( S/I, \ \lambda \right) = \frac{\sum_{i = 0}^{p} (-1)^{i} \sum_{j \in \mathbb{Z}} \beta_{i, i + j}(S/I) \lambda^{i + j} }{(1 - \lambda)^{n}} 
= \frac{h(S/I, \lambda) \cdot (1 - \lambda)^{n - \dim S/I}}{(1 - \lambda)^{n}}, 
\]
it follows that $\deg h\left(S/I, \ \lambda \right) \leq p + \reg\left( S/I \right) - n + \dim S/I$.  Furthermore, since 
%
%
$n - p = \depth(S/I)$ by Auslander--Buchsbaum Theorem, 
the inequality (\ref{eq:srde}) follows.  
In addition, the equality 
\[
(*) \ \ \deg h\left(S/I, \ \lambda \right) - \reg\left( S/I \right) = \dim S/I - \depth\left( S/I \right)
\] 
holds if and only if $\beta_{p, p + \reg (S/I)} (S/I) \neq 0$, in other words, 
if and only if 
$S/I$ has a unique extremal Betti number (\cite[Definition 4.3.13]{HH}). 
In particular, the equality $(*)$ holds 
if $S/I$ is Cohen--Macaulay by \cite[Lemma 3]{BiHe} 
or $I$ has a pure resolution (\cite[p.~153]{BH}).



\par
Let $G$ be a finite simple graph (i.e. a graph with no loop and no multiple edge) on the vertex set 
$V(G) = \{x_{1}, x_{2}, \ldots, x_{n}\}$ and its edge set $E(G)$. 
Set $S = K[V(G)]$. 
The {\em edge ideal} of $G$ is
\[
I(G) = \left(x_{i}x_{j}\, : \,  \{x_{i}, x_{j}\} \in E(G)\right) \subset S. 
\]
It is natural to ask for which graph $G$, its edge ideal $I(G)$ satisfies 
$a(S/I(G)) = 0$ or the equality $(\ast)$. 
In the present paper we focus on Cameron--Walker graphs. 
Let us recall the definition of a Cameron--Walker graph. 
Let $im(G)$ (resp.\  $m(G)$) denote the induced matching number 
(resp.\  matching number) of $G$, 
see \cite[p.258]{HHKO}. 
Then 
for any finite simple graph $G$, one has 
\begin{equation}
  \label{im-r-m}
im(G) \leq \reg\left( S/I(G) \right) \leq m(G) 
\end{equation}
by virtue of \cite[Theorem 6.7]{HaVanTuyl} and \cite[Lemma 2.2]{K}. 
Cameron and Walker \cite[Theorem 1]{CW} (see also \cite[Remark 0.1]{HHKO}) 
characterized a finite connected simple graph $G$ satisfying $im (G) = m(G)$. 
A {\em Cameron--Walker graph} $G$ is a graph satisfying $im (G) = m(G)$ 
which is neither a star graph nor a star triangle; 
see Section \ref{sec:a-invariant} 
for more detail. 
In \cite{HHKO, TNT}, Cameron--Walker graphs have been studied from a viewpoint of commutative algebra. 

\par
In the present paper, we first prove $a(S/I(G)) = 0$ 
for every Cameron--Walker graph $G$ (Theorem \ref{CW}) 
in Section \ref{sec:a-invariant}. 
We next give a classification of Cameron--Walker graphs $G$ 
whose edge ideal $I(G)$ satisfies the equality $(*)$ 
(Theorem \ref{Main}) 
in Section \ref{sec:(ast)}. 
We also provide some classes of graphs 
other than Cameron--Walker graphs satisfying 
$(\ast)$ (Proposition \ref{other-ast}). 
In general, there is no relationship between the degree 
of the $h$-polynomial and the regularity even for edge ideals; 
see \cite{HMVT}. 
However we prove in Section \ref{sec:application} that 
for a Cameron--Walker graph $G$, 
the inequality $\deg h\left(S/I(G), \lambda \right) \geq \reg (S/I(G))$ holds. 
Moreover we characterize the Cameron--Walker graphs $G$ which satisfy 
the equality (Theorem \ref{s=r}). 

\section{$a$-invariant of Cameron--Walker graphs}
\label{sec:a-invariant}
In this section, we show
\begin{Theorem}\label{CW}
Let $G$ be a Cameron--Walker graph. 
Then $a(K[V(G)]/I(G)) = 0$. 
\end{Theorem}

We first recall the definition of a Cameron--Walker graph. 
Let $G$ be a finite simple graph on the vertex set 
$V(G)$ with the edge set $E(G)$. 
We call a subset $\mathcal{M} \subset E(G)$ a {\em matching} of $G$ 
if $e \cap e' = \emptyset$ for any $e, e' \in \mathcal{M}$ with $e \neq e'$. 
A matching $\mathcal{M}$ of $G$ is called an {\em induced matching} of $G$ if 
for $e, e' \in \mathcal{M}$ with $e \neq e'$, there is no edge $f \in E(G)$ 
with $e \cap f \neq \emptyset$ and $e' \cap f \neq \emptyset$. 
The {\em matching number} $m(G)$ of $G$ is the maximum cardinality 
of the matchings of $G$. 
Also the {\em induced matching number} $im(G)$ of $G$ is the maximum cardinality 
of the induced matchings of $G$. 
As noted in Introduction, the inequalities 
$im(G) \leq \reg \left(K[V(G)]/I(G)\right) \leq m(G)$ hold. 
By virtue of \cite[Theorem 1]{CW} together with \cite[Remark 0.1]{HHKO}, 
the equality $im(G) = m(G)$ holds 
if and only if $G$ is one of the following graphs: 
\begin{itemize}
	\item a star graph, i.e. a graph joining some paths of length $1$ at one common vertex (see Figure \ref{fig:Star}); 
	\item a star triangle, i.e. a graph joining some triangles at one common vertex (see Figure \ref{fig:StarTriangle}); 
	\item a connected finite graph consisting of a connected bipartite graph with vertex partition 
	$\{v_{1}, \ldots, v_{m}\} \cup \{w_{1}, \ldots, w_{n}\}$ such that there is at least one leaf edge 
	attached to each vertex $v_{i}$ and that there may be possibly some pendant triangles 
	attached to each vertex $w_{j}$. 
	Here a leaf edge is an edge meeting a vertex of degree $1$ and a pendant triangle is a triangle 
	whose two vertices have degree $2$ and the rest vertex has degree more than $2$. 
\end{itemize}
We say that a finite connected simple graph $G$ is {\em Cameron--Walker} 
if $im(G) = m(G)$ and if $G$ is neither a star graph nor a star triangle. 
\begin{Remark}\normalfont
  One can consider a star graph $G$ with $|V(G)| \geq 3$ 
  as a Cameron--Walker graph consisting of bipartite graph $\mathcal{K}_{1,1}$ 
  with some leaf edges and without pendant triangle. 
  Hence claims for Cameron--Walker graph in the below are also true 
  for such a star graph. 
\end{Remark}
Note that for a Cameron--Walker graph $G$, 
the regularity of $K[V(G)]/I(G)$ is equal to $im(G)$ (equivalently, $m(G)$). 

\par
Let $G$ be a Cameron--Walker graph. 
In what follows we use the following labeling on vertices of $G$; 
see Figure \ref{fig:CameronWalkerGraph}: 
\[
V(G) = \bigcup_{i = 1}^{m} \left\{ x^{(i)}_{1}, \ldots, x^{(i)}_{s_{i}} \right\} \cup \{ v_{1}, \ldots, v_{m} \} \cup \{ w_{1}, \ldots, w_{n} \} \cup \left\{ \bigcup_{j = 1}^{n} \bigcup_{\ell = 1}^{t_{j}} \left\{ y^{(j)}_{\ell, 1}, y^{(j)}_{\ell, 2} \right\} \right\},
\] 
where 
$\{ v_1, \ldots, v_m \} \cup \{ w_1, \ldots, w_n \}$ is a vertex partition 
of a connected bipartite subgraph of $G$, 
$x_k^{(i)}$ ($i=1, \ldots, m$; $k=1, \ldots, s_i$) is a vertex such that 
$\left\{ v_i, x_k^{(i)} \right\}$ is a leaf edge, 
and $y_{\ell, 1}^{(j)}, y_{\ell, 2}^{(j)}$ 
($j=1, \ldots, n$; $\ell = 1, \ldots, t_j$) 
are vertices which together with $w_{j}$ form a pendant triangle. 
Note that $s_{i} \geq 1$ and $t_{j} \geq 0$. 

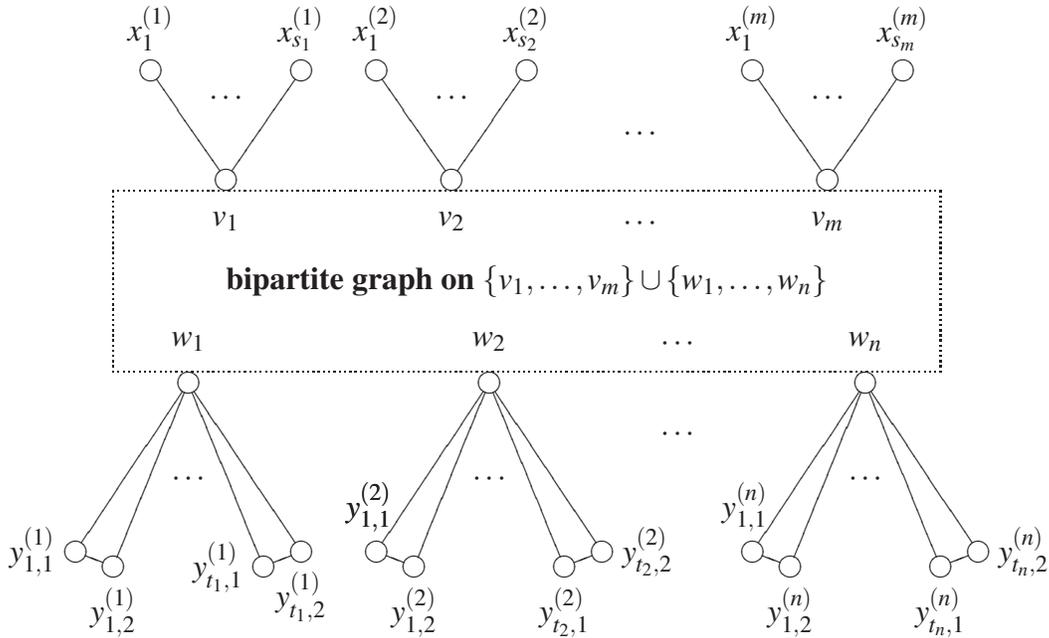
\begin{figure}[htbp]
  \centering
\begin{xy}
	\ar@{} (0,0);(20, 12)  = "A";
	\ar@{.} "A";(130, 12)  = "B";
	\ar@{.} "A";(20, -12)  = "C";
	\ar@{.} "B";(130, -12)  = "D";
	\ar@{.} "C";"D";
	\ar@{} (0,0);(75,0) *{\text{\bf{bipartite graph on} $\{v_{1}, \ldots, v_{m}\} \cup \{w_{1}, \ldots, w_{n}\}$}};
	\ar@{} "A";(35, 13.5)  *\cir<4pt>{} = "E1";
	\ar@{-} "E1";(25, 28) *++!D{x^{(1)}_{1}} *\cir<4pt>{};
	\ar@{-} "E1";(45, 28) *++!D{x^{(1)}_{s_{1}}} *\cir<4pt>{};
	\ar@{} (0,0); (35, 28) *++!U{\cdots}
	\ar@{} "A";(65, 13.5)  *\cir<4pt>{} = "E2";
	\ar@{-} "E2";(55, 28) *++!D{x^{(2)}_{1}} *\cir<4pt>{};
	\ar@{-} "E2";(75, 28) *++!D{x^{(2)}_{s_{2}}} *\cir<4pt>{};
	\ar@{} (0,0); (65, 28) *++!U{\cdots}
	\ar@{} "A";(115, 13.5)  *\cir<4pt>{} = "E";
	\ar@{-} "E";(105, 28) *++!D{x^{(m)}_{1}} *\cir<4pt>{};
	\ar@{-} "E";(125, 28) *++!D{x^{(m)}_{s_{m}}} *\cir<4pt>{};
	\ar@{} (0,0); (115, 28) *++!U{\cdots}
	\ar@{} (0,0);(35,8) *{\text{$v_{1}$}};
	\ar@{} (0,0);(65,8) *{\text{$v_{2}$}};
	\ar@{} (0,0);(90,4) *++!D{\cdots};
        \ar@{} (0,0);(115,8) *{\text{$v_{m}$}};
	\ar@{} (0,0);(90,16) *++!D{\cdots};
	\ar@{} "E1";(30, -13.5)  *\cir<4pt>{} = "F10";
	\ar@{-} "F10";(15, -36) *++!R{y^{(1)}_{1,1}} *\cir<4pt>{} = "F11";
	\ar@{-} "F10";(20, -38) *++!U{y^{(1)}_{1,2}} *\cir<4pt>{} = "F12";
	\ar@{-} "F11";"F12";
	\ar@{-} "F10";(40, -38) *++!R{y^{(1)}_{t_{1},1}} *\cir<4pt>{} = "Ft1";
	\ar@{-} "F10";(45, -36) *++!U{y^{(1)}_{t_{1},2}} *\cir<4pt>{} = "Ft2";
	\ar@{-} "Ft1";"Ft2";
	\ar@{} "E1";(70, -13.5)  *\cir<4pt>{} = "F20";
	\ar@{-} "F20";(55, -36) *\cir<4pt>{} = "F21";
	\ar@{} (0,0);(54,-30) *{\text{$y^{(2)}_{1,1}$}};
	\ar@{-} "F20";(60, -38) *++!U{y^{(2)}_{1,2}} *\cir<4pt>{} = "F22";
	\ar@{-} "F21";"F22";
	\ar@{-} "F20";(80, -38) *++!U{y^{(2)}_{t_{2},1}} *\cir<4pt>{} = "Ftt1";
	\ar@{-} "F20";(85, -36) *++!L{y^{(2)}_{t_{2},2}} *\cir<4pt>{} = "Ftt2";
	\ar@{-} "Ftt1";"Ftt2";
	\ar@{} "E1";(120, -13.5)  *\cir<4pt>{} = "Fn0";
	\ar@{-} "Fn0";(105, -36) *\cir<4pt>{} = "Fn1";
	\ar@{} (0,0);(104,-30) *{\text{$y^{(n)}_{1,1}$}};
	\ar@{-} "Fn0";(110, -38) *++!U{y^{(n)}_{1,2}} *\cir<4pt>{} = "Fn2";
	\ar@{-} "Fn1";"Fn2";
	\ar@{-} "Fn0";(130, -38) *++!U{y^{(n)}_{t_{n},1}} *\cir<4pt>{} = "Fnn1";
	\ar@{} (0,0);(54,-30) *{\text{$y^{(2)}_{1,1}$}};
	\ar@{-} "Fn0";(135, -36) *++!L{y^{(n)}_{t_{n},2}} *\cir<4pt>{} = "Fnn2";
	\ar@{-} "Fnn1";"Fnn2";
	\ar@{} (0,0);(30,-8) *{\text{$w_{1}$}};
	\ar@{} (0,0);(70,-8) *{\text{$w_{2}$}};
	\ar@{} (0,0);(95,-12) *++!D{\cdots};
	\ar@{} (0,0);(120,-8) *{\text{$w_{n}$}};
	\ar@{} (0,0);(95,-24) *++!D{\cdots};
	\ar@{} (0,0);(30,-30) *++!D{\cdots};
	\ar@{} (0,0);(70,-30) *++!D{\cdots};
	\ar@{} (0,0);(120,-30) *++!D{\cdots};
\end{xy}

\bigskip

  \caption{Cameron--Walker graph}
  \label{fig:CameronWalkerGraph}
\end{figure}

\par
We prove Theorem \ref{CW} by showing 
\begin{Proposition}
  \label{s=d}
  Let $G$ be a Cameron--Walker graph as in Figure \ref{fig:CameronWalkerGraph}. 
  Then
  \begin{equation}
    \label{eq:s=d}
    \deg h\left(K[V(G)]/I(G), \ \lambda \right) = \dim K[V(G)]/I(G) = \sum_{i = 1}^{m} s_{i} + \sum_{j = 1}^{n} \max\left\{ t_{j}, 1 \right\}. 
  \end{equation}
\end{Proposition}

Before giving a proof of Proposition \ref{s=d}, 
several lemmata will be prepared.
Let $I \subset S$ be a monomial ideal of $S$ and let $x$ be a variable of $S$ 
which appears in some monomial belonging to the unique minimal system of 
monomial generators of $I$. 
Then, by the additivity of Hilbert series on the exact sequence 
$0 \to S/I : (x) (-1) \xrightarrow{\ \cdot x \ } S/I \to S/I + (x) \to 0$, 
one has
 
\begin{Lemma}\label{HilbertSeries}
\[ 
H\left(S/I, \ \lambda\right) = H\left(S/I + (x), \ \lambda \right) + \lambda \cdot H\left(S/I : (x), \ \lambda \right). 
\] 
\end{Lemma}

Let $G$ be a finite simple graph on the vertex set $V(G) = \{x_{1}, \ldots, x_{n}\}$ with the edge set $E(G)$. 
For $W \subset V(G)$, the {\em induced subgraph} $G_{W}$ is the subgraph of $G$ 
such that $V(G_{W}) = W$ and $E(G_{W}) = \{ \{x_{i}, x_{j}\} \in E(G) : x_{i}, x_{j} \in W \}$. 
For $x_{v} \in V(G)$, let $N_{G}(x_{v})$ denote the neighborhood of $x_{v}$ 
and let $N_{G}[x_{v}] = N_{G}(x_{v}) \cup \{x_{v}\}$. 
Then $I(G) + (x_{v}) = (x_{v}) + I\left(G_{V(G) \setminus \{x_{v}\}}\right)$ and 
$I(G) : (x_{v}) = \left(x_{i} : x_{i} \in N_{G}(x_{v})\right) + I\left(G_{V(G) \setminus N_{G}[x_{v}]}\right)$.   
Hence
\[
\frac{K[V(G)]}{I(G) + (x_{v})} \cong \frac{K[V(G) \setminus \{x_{v}\}]}{I\left(G_{V(G) \setminus \{x_{v}\}}\right)}, 
\]
\[
\frac{K[V(G)]}{I(G) : (x_{v})} \cong \frac{K[V(G) \setminus N_{G}[x_{v}]\ ]}{I\left(G_{V(G) \setminus N_{G}[x_{v}]}\right)} \otimes_{K} K[x_{v}]. 
\]
Thus, by virtue of Lemma \ref{HilbertSeries}, it follows that  

\begin{Lemma}\label{HSofEdgeIdeal}
\begin{eqnarray*}
& & H\left({K[V(G)]/I(G), \ \lambda}\right)\\
&=& H\left({\frac{K[V(G) \setminus \{x_{v}\}]}{I\left(G_{V(G) \setminus \{x_{v}\}}\right)},  \ \lambda}\right) + H\left({\frac{K[V(G) \setminus N_{G}[x_{v}]\ ]}{I\left(G_{V(G) \setminus N_{G}[x_{v}]}\right)}, \ \lambda}\right) \cdot \frac{\lambda}{1 - \lambda}. 
\end{eqnarray*}
\end{Lemma}

The following lemma is somewhat technical. 

\begin{Lemma}\label{Lemma}
Let $G$ be a finite simple graph and let $x_{v} \in V(G)$. 
Assume that 
\begin{enumerate}
	\item $\displaystyle \deg h\left({\frac{K[V(G) \setminus \{x_{v}\}]}{I\left(G_{V(G) \setminus \{x_{v}\}}\right)},  \ \lambda}\right) < \dim \frac{K[V(G) \setminus \{x_{v}\}]}{I\left(G_{V(G) \setminus \{x_{v}\}}\right)} =: d$\ $;$ 
	\item $\displaystyle \deg h\left({\frac{K[V(G) \setminus N_{G}[x_{v}]\ ]}{I\left(G_{V(G) \setminus N_{G}[x_{v}]}\right)}, \ \lambda}\right) = \dim \frac{K[V(G) \setminus N_{G}[x_{v}]\ ]}{I\left(G_{V(G) \setminus N_{G}[x_{v}]}\right)} =: d'$\ $;$ 
	\item $d > d'$. 
\end{enumerate}
Then $\displaystyle \deg h\left(K[V(G)]/I(G), \ \lambda \right) = \dim K[V(G)]/I(G) = d$. 
\end{Lemma}
\begin{proof}
It follows from Lemma \ref{HSofEdgeIdeal}. 
\end{proof}

By using Lemma \ref{HSofEdgeIdeal} again, 
one has the Hilbert series of $K[V(G)]/I(G)$ 
when $G$ is a star graph or a star triangle. 
For $s \geq 1$, we denote by $G^{{\rm star}(x_{v})}_{s}$, 
the star graph joining $s$ paths of length $1$ at the common vertex $x_{v}$; 
see Figure \ref{fig:Star}. 

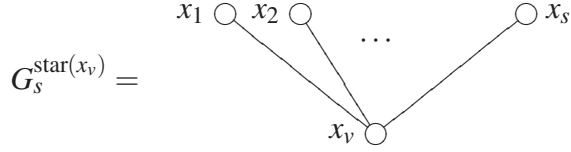
\begin{figure}[htbp]
  \centering
\bigskip

\begin{xy}
	\ar@{} (0,0);(35,-8) *{\text{$G^{{\rm star}(x_{v})}_{s} =$}};
	\ar@{} (0,0);(75, -16)  *++! R{x_{v}} *\cir<4pt>{} = "C"
	\ar@{-} "C";(55, 0) *++! R{x_{1}} *\cir<4pt>{} = "D";
	\ar@{-} "C";(65, 0) *++!R{x_{2}} *\cir<4pt>{} = "E";
	\ar@{} "C"; (75, 0) *++!U{\cdots}
	\ar@{-} "C";(95, 0) *++!L{x_{s}} *\cir<4pt>{} = "F";
\end{xy}

\bigskip

  \caption{The star graph $G^{{\rm star}(x_{v})}_{s}$}
  \label{fig:Star}
\end{figure}

\begin{Lemma}
\label{Star}
Let $s \geq 1$ be an integer. Then
%
\begin{displaymath}
  H \left(K[V(G^{{\rm star}(x_{v})}_{s})]/I(G^{{\rm star}(x_{v})}_{s}), \ \lambda \right) 
  = \frac{1 + \lambda (1 - \lambda)^{s - 1}}{(1 - \lambda)^{s}}. 
\end{displaymath} 
In particular, 
\[
\deg h\left(K[V(G^{{\rm star}(x_{v})}_{s})]/I(G^{{\rm star}(x_{v})}_{s}), \ \lambda \right) = \dim K[V(G^{{\rm star}(x_{v})}_{s})]/I(G^{{\rm star}(x_{v})}_{s}) = s.
\] 
\end{Lemma}

For $t \geq 1$, we denote by $G^{{\triangle}(x_{v})}_{t}$, 
the star triangle joining $t$ triangles at the common vertex $x_{v}$; 
see Figure \ref{fig:StarTriangle}. 

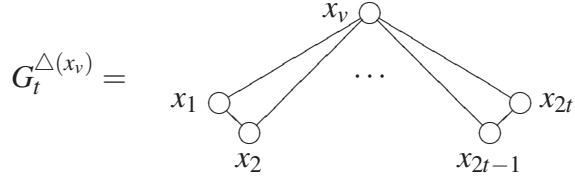
\begin{figure}[htbp]
  \centering
\bigskip

\begin{xy}
	\ar@{} (0,0);(35,-8) *{\text{$G^{\triangle (x_{v})}_{t} =$}};
	\ar@{} (0,0);(75, 0) *++! R{x_{v}} *\cir<4pt>{} = "C"
	\ar@{-} "C";(55, -12) *++!R{x_{1}} *\cir<4pt>{} = "D";
	\ar@{-} "C";(59, -16) *++!U{x_{2}} *\cir<4pt>{} = "E";
	\ar@{-} "D";"E";
	\ar@{} "C"; (75, -5) *++!U{\cdots}
	\ar@{-} "C";(95, -12) *++!L{x_{2t}} *\cir<4pt>{} = "F";
	\ar@{-} "C";(91, -16) *++!U{x_{2t-1}} *\cir<4pt>{} = "G";
	\ar@{-} "F";"G"; 
\end{xy}

\bigskip

  \caption{The star triangle $G^{{\triangle}(x_{v})}_{t}$}
  \label{fig:StarTriangle}
\end{figure}

\begin{Lemma}
\label{StarTriangle}
Let $t \geq 1$ be an integer. Then
%
\[
H\left(K[V(G^{\triangle (x_{v})}_{t})]/I(G^{\triangle (x_{v})}_{t}), \ \lambda \right) = \frac{(1 + \lambda)^{t} + \lambda (1 - \lambda)^{t - 1}}{(1 - \lambda)^{t}}. 
\]
In particular, 
\begin{equation*}
	\deg h\left(K[V(G^{\triangle (x_{v})}_{t})]/I(G^{\triangle (x_{v})}_{t}), \ \lambda \right) = \begin{cases} 
		t & \text{(t $:$ odd)} \\
		t - 1 & \text{(t $:$ even)} \end{cases} 
\end{equation*}
and $\displaystyle \dim K[V(G^{\triangle (x_{v})}_{t})]/I(G^{\triangle (x_{v})}_{t}) = t$. 
\end{Lemma}

We also use the following lemmata. 
\begin{Lemma}[{\cite[Lemma 1.5(i)]{HT}}]
\label{HSofTensorProduct} 
Let $S_{1}$ and $S_{2}$ be 
polynomial rings over a field $K$.  
Let $I_{1}$ be a nonzero homogeneous ideal of $S_{1}$ and $I_2$ that of $S_2$.  
Write $S$ for $S_{1} \otimes_{K} S_{2}$ 
and regard $I_{1} + I_{2}$ as homogeneous ideals of $S$. Then
\[
H\left(S/I_{1} + I_{2}, \ \lambda\right) = H\left(S_{1}/I_{1}, \ \lambda\right) \cdot  H\left(S_{2}/I_{2}, \ \lambda\right).  
\]
In particular, 
\[
\deg h\left(S/I_{1} + I_{2}, \ \lambda\right) = \deg h\left(S_{1}/I_{1}, \ \lambda\right) + \deg h\left(S_{2}/I_{2}, \ \lambda\right), 
\]
\[
\dim S/I_{1} + I_{2} = \dim S_{1}/I_{1} + \dim S_{2}/I_{2}. 
\]
\end{Lemma}

Let $G$ be a disconnected graph whose connected components are 
$G_{1}, \ldots, G_{r}$. 
Then $I(G) = \sum_{i = 1}^{r} I(G_{i})$. 
Thus, by virtue of Lemma \ref{HSofTensorProduct}, one has 

\begin{Lemma}
\label{disconnected}
Under the notation as above, 
\[
\deg h\left(K[V(G)]/I(G), \ \lambda \right) = \sum_{i = 1}^{r} \deg h\left(K[V(G_{i}]/I(G_{i}), \ \lambda \right), 
\]
\[
\dim K[V(G)]/I(G) = \sum_{i = 1}^{r} \dim K[V(G_{i})]/I(G_{i}),  
\]
here we regard $K[V(G_{i})]/I(G_{i})$ as a 1-dimensional polynomial ring 
if $G_{i}$ is an isolated vertex. 
\end{Lemma}  

Now we are in the position to prove Proposition \ref{s=d}. 

\begin{proof}[Proof of Proposition \ref{s=d}] 
Let $G$ be a Cameron--Walker graph as in Figure \ref{fig:CameronWalkerGraph}.  
We prove the equality (\ref{eq:s=d}) 
by using induction on $m + n$. 

\par
First, we assume that $m + n = 2$.  
Then $m = n = 1$. 
If $t_{1} = 0$, then $G = G^{{\rm star}(v_{1})}_{s_{1} + 1}$. 
Hence the equality (\ref{eq:s=d}) follows by Lemma \ref{Star}. 
Next assume $t_{1} > 0$ . 
We will show 
\[
\deg h\left(K[V(G)]/I(G), \ \lambda \right) = \dim K[V(G)]/I(G) = s_{1} + t_{1}.
\] 
Note that 
\begin{itemize}
\item  $G_{V(G) \setminus \{v_{1}\}}$ consists of $s_{1}$ isolated vertices 
  and a star triangle $G^{\triangle (w_{1})}_{t_{1}}$; 
\item $G_{V(G) \setminus N_{G}[v_{1}]}$ consists of $t_{1}$ star graphs 
  $G^{{\rm star}(y^{(1)}_{1, 1})}_{1}, \ldots, G^{{\rm star}(y^{(1)}_{t_{1}, 1})}_{1}$; 
\end{itemize}
see Figure \ref{fig:m=n=1}. 

\begin{figure}[htbp]
  \centering

\bigskip

\begin{xy}
	\ar@{} (0,0);(40, 0) *\dir<4pt>{*} = "B";
	\ar@{.} "B";(40, -8) *++!L{w_{1}} *\cir<4pt>{} = "C";
	\ar@{-} "C";(20, -20) *++!R{y_{1, 1}^{(1)}} *\cir<4pt>{} = "D";
	\ar@{-} "C";(24, -24) *++!U{y_{1, 2}^{(1)}} *\cir<4pt>{} = "E";
	\ar@{-} "D";"E";
	\ar@{} (0,0); (40, -13) *++!U{\cdots}
	\ar@{-} "C";(60, -20) *++!L{y_{t_{1}, 2}^{(1)}} *\cir<4pt>{} = "F";
	\ar@{-} "C";(56, -24) *++!U{y_{t_{1}, 1}^{(1)}} *\cir<4pt>{} = "G";
	\ar@{-} "F";"G"; 
	\ar@{.} "B";(25, 15) *++!R{x_{1}^{(1)}} *\cir<4pt>{};
	\ar@{.} "B";(55, 15) *++!R{x_{s_{1}}^{(1)}} *\cir<4pt>{};
	\ar@{} (0,0); (40, 12) *++!U{\cdots}
	\ar@{} (0,0);(110, 0) *\dir<4pt>{*} = "B2";
	\ar@{.} "B2";(110, -8)  *\dir<4pt>{*} = "C2";
	\ar@{.} "C2";(90, -20) *++!R{y_{1, 1}^{(1)}} *\cir<4pt>{} = "D2";
	\ar@{.} "C2";(94, -24) *++!U{y_{1, 2}^{(1)}} *\cir<4pt>{} = "E2";
	\ar@{-} "D2";"E2";
	\ar@{} (0,0); (110, -13) *++!U{\cdots}
	\ar@{.} "C2";(130, -20) *++!L{y_{t_{1}, 2}^{(1)}} *\cir<4pt>{} = "F2";
	\ar@{.} "C2";(126, -24) *++!U{y_{t_{1}, 1}^{(1)}} *\cir<4pt>{} = "G2";
	\ar@{-} "F2";"G2"; 
	\ar@{.} "B2";(95, 15)  *\dir<4pt>{*};
	\ar@{.} "B2";(125, 15)  *\dir<4pt>{*};
	\ar@{} (0,0); (110, 12) *++!U{\cdots}
	\ar@{} (0,0);(40,-38) *{\text{$G_{V(G) \setminus \{v_{1}\}}$}};
	\ar@{} (0,0);(110,-38) *{\text{$G_{V(G) \setminus N_{G}[v_{1}]}$}};
\end{xy}
  \caption{$G_{V(G) \setminus \{v_{1}\}}$ (left) and $G_{V(G) \setminus N_{G}[v_{1}]}$ (right)}
  \label{fig:m=n=1}
\end{figure}
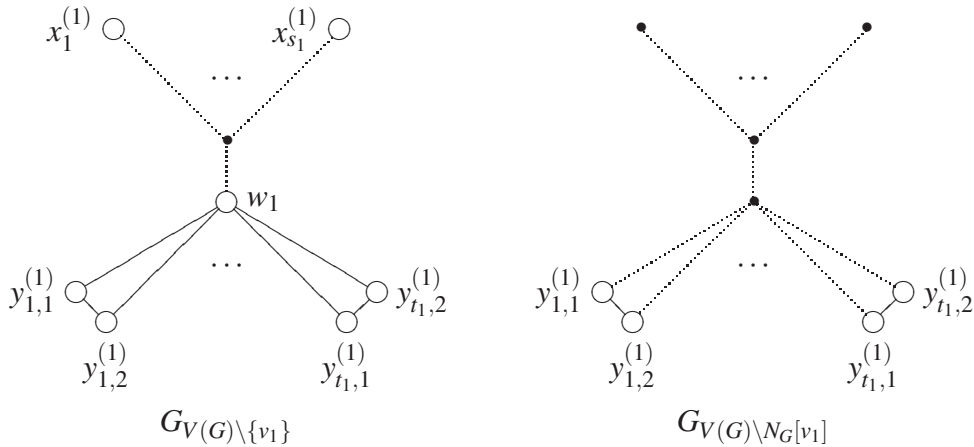
\bigskip

Hence, by using Lemmata \ref{Star}, \ref{StarTriangle} and \ref{HSofTensorProduct}, one has 

\[
H\left({\frac{K[V(G) \setminus \{v_{1}\}]}{I\left(G_{V(G) \setminus \{v_{1}\}}\right)},  \ \lambda}\right) = \frac{(1 + \lambda)^{t_{1}} + \lambda(1 - \lambda)^{t_{1} - 1}}{(1 - \lambda)^{s_{1} + t_{1}}}
\]
and 
\[
H\left({\frac{K[V(G) \setminus N_{G}[v_{1}]\ ]}{I\left(G_{V(G) \setminus N_{G}[v_{1}]}\right)}, \ \lambda}\right) = \frac{(1 + \lambda)^{t_{1}}}{(1 - \lambda)^{t_{1}}}. 
\]
Thus, by virtue of Lemma \ref{HSofEdgeIdeal}, it follows that  
\begin{eqnarray*}
H(K[V(G)]/I(G), \ \lambda) &=& \frac{(1 + \lambda)^{t_{1}} + \lambda(1 - \lambda)^{t_{1} - 1}}{(1 - \lambda)^{s_{1} + t_{1}}} + \frac{(1 + \lambda)^{t_{1}}}{(1 - \lambda)^{t_{1}}} \cdot \frac{\lambda}{1 - \lambda} \\
&=& \frac{(1 + \lambda)^{t_{1}} + \lambda(1 - \lambda)^{t_{1} - 1} + \lambda(1 + \lambda)^{t_{1}}(1 - \lambda)^{s_{1} - 1}}{(1 - \lambda)^{s_{1} + t_{1}}}. 
\end{eqnarray*}
Therefore one has $\deg h\left(K[V(G)]/I(G), \ \lambda \right) = \dim K[V(G)]/I(G) = s_{1} + t_{1}$, as desired.
%
%

\bigskip

Next, we assume that $m + n > 2$. 

{\bf (First Step.)} Let $m = 1$ and $n > 1$. 
Suppose that there exists $1 \leq \ell \leq n$ such that $t_{\ell} = 0$. 
We may assume $\ell = n$. 
Then we will show 
\[
\deg h\left(K[V(G)]/I(G), \ \lambda \right) = \dim K[V(G)]/I(G) = s_{1} + \sum_{j = 1}^{n - 1}  \max\{t_{j}, 1\} + 1.  
\]
Since $t_{n} = 0$, $\{ v_{1}, w_{n} \}$ is a leaf edge. 
Hence we can regard $G$ as a Cameron--Walker graph such that its bipartite part is 
the star graph $G^{{\rm star}(v_{1})}_{n - 1}$ and 
the vertex $v_{1}$ has $s_{1} + 1$ leaf edges. 
Thus, by induction hypothesis, one has
\[
\deg h\left(K[V(G)]/I(G), \ \lambda \right) = \dim K[V(G)]/I(G) = s_{1} + 1 + \sum_{j = 1}^{n - 1} \max\{t_{j}, 1\},  
\] 
as desired. 

Next, suppose that $t_{j} > 0$ for all $1 \leq j \leq n$. 
We will show 
\[
\deg h\left(K[V(G)]/I(G), \ \lambda \right) = \dim K[V(G)]/I(G) = s_{1} + \sum_{j = 1}^{n} t_{j}.
\] 
Note that 
\begin{itemize}
\item $G_{V(G) \setminus \{v_{1}\}}$ consists of $s_{1}$ isolated vertices and 
  $n$ star triangles $G^{\triangle (w_{1})}_{t_{1}}, \ldots, G^{\triangle (w_{n})}_{t_{n}}$, 
\item $G_{V(G) \setminus N_{G}[v_{1}]}$ consists of $\sum_{j = 1}^{n} t_{j}$ star graphs 
  $G^{{\rm star}(y^{(k)}_{\ell, 1})}_{1}$ for $1 \leq k \leq n$ and $1 \leq \ell \leq t_{k}$; 
\end{itemize}
see Figure \ref{fig:m=1;n>1}. 

\begin{figure}[htbp]
  \centering

  \bigskip

\begin{xy}
	\ar@{} (0,0);(35,-65) *{\text{$G_{V(G) \setminus \{v_{1}\}}$}};	
	\ar@{} (0,0);(35, 0)  *\dir<4pt>{*} = "E1";
	\ar@{.} "E1";(25, 10) *++!D{x_{1}} *\cir<4pt>{};
	\ar@{.} "E1";(45, 10) *++!D{x_{s_{1}}} *\cir<4pt>{};
	\ar@{} (0,0); (35, 6) *++!D{\cdots}
	\ar@{} (0,0); (35, -16) *++!D{\cdots}
	\ar@{} "E1";(20, -13.5)  *\cir<4pt>{} = "F10";
	\ar@{-} "F10";(8, -46) *\cir<4pt>{} = "F11";
	\ar@{} (0,0);(6,-40) *{\text{$y^{(1)}_{1,1}$}};
	\ar@{-} "F10";(13, -48) *++!U{y^{(1)}_{1,2}} *\cir<4pt>{} = "F12";
	\ar@{-} "F11";"F12";
	\ar@{-} "F10";(27, -48) *++!R{y^{(1)}_{t_{1},1}} *\cir<4pt>{} = "Ft1";
	\ar@{-} "F10";(32, -46) *++!U{y^{(1)}_{t_{1},2}} *\cir<4pt>{} = "Ft2";
	\ar@{-} "Ft1";"Ft2";
	\ar@{} "E1";(50, -13.5)  *\cir<4pt>{} = "Fn0";
	\ar@{-} "Fn0";(38, -46) *\cir<4pt>{} = "Fn1";
	\ar@{} (0,0);(36,-40) *{\text{$y^{(n)}_{1,1}$}};
	\ar@{-} "Fn0";(43, -48) *++!U{y^{(n)}_{1,2}} *\cir<4pt>{} = "Fn2";
	\ar@{-} "Fn1";"Fn2";
	\ar@{-} "Fn0";(57, -48) *++!U{y^{(n)}_{t_{n},1}} *\cir<4pt>{} = "Fnn1";
	\ar@{-} "Fn0";(62, -46) *++!L{y^{(n)}_{t_{n},2}} *\cir<4pt>{} = "Fnn2";
	\ar@{-} "Fnn1";"Fnn2";
	\ar@{} (0,0);(20,-8) *{\text{$w_{1}$}};
	\ar@{} (0,0);(50,-8) *{\text{$w_{n}$}};
	\ar@{} (0,0);(20,-40) *++!D{\cdots};
	\ar@{} (0,0);(50,-40) *++!D{\cdots};
	\ar@{.} "F10";"E1";
	\ar@{.} "Fn0";"E1";
	
	\ar@{} (0,0);(115,-65) *{\text{$G_{V(G) \setminus N_{G}[v_{1}]}$}};	
	\ar@{} (0,0);(115, 0)  *\dir<4pt>{*} = "E1r";
	\ar@{.} "E1r";(105, 10)  *\dir<4pt>{*};
	\ar@{.} "E1r";(125, 10)  *\dir<4pt>{*};
	\ar@{} (0,0); (115, 6) *++!D{\cdots}
	\ar@{} (0,0); (115, -16) *++!D{\cdots}
	\ar@{} "E1r";(100, -13.5)  *\dir<4pt>{*} = "F10r";
	\ar@{.} "F10r";(88, -46) *\cir<4pt>{} = "F11r";
	\ar@{} (0,0);(86,-40) *{\text{$y^{(1)}_{1,1}$}};
	\ar@{.} "F10r";(93, -48) *++!U{y^{(1)}_{1,2}} *\cir<4pt>{} = "F12r";
	\ar@{-} "F11r";"F12r";
	\ar@{.} "F10r";(107, -48) *++!R{y^{(1)}_{t_{1},1}} *\cir<4pt>{} = "Ft1r";
	\ar@{.} "F10r";(112, -46) *++!U{y^{(1)}_{t_{1},2}} *\cir<4pt>{} = "Ft2r";
	\ar@{-} "Ft1r";"Ft2r";
	\ar@{} "E1r";(130, -13.5)  *\dir<4pt>{*} = "Fn0r";
	\ar@{.} "Fn0r";(118, -46) *\cir<4pt>{} = "Fn1r";
	\ar@{} (0,0);(116,-40) *{\text{$y^{(n)}_{1,1}$}};
	\ar@{.} "Fn0r";(123, -48) *++!U{y^{(n)}_{1,2}} *\cir<4pt>{} = "Fn2r";
	\ar@{-} "Fn1r";"Fn2r";
	\ar@{.} "Fn0r";(137, -48) *++!U{y^{(n)}_{t_{n},1}} *\cir<4pt>{} = "Fnn1r";
	\ar@{.} "Fn0r";(142, -46) *++!L{y^{(n)}_{t_{n},2}} *\cir<4pt>{} = "Fnn2r";
	\ar@{-} "Fnn1r";"Fnn2r";
	\ar@{} (0,0);(100,-40) *++!D{\cdots};
	\ar@{} (0,0);(130,-40) *++!D{\cdots};
	\ar@{.} "F10r";"E1r";
	\ar@{.} "Fn0r";"E1r";

\end{xy}

\bigskip

  \caption{$G_{V(G) \setminus \{v_{1}\}}$ (left) and $G_{V(G) \setminus N_{G}[v_{1}]}$ (right)}
  \label{fig:m=1;n>1}
\end{figure}
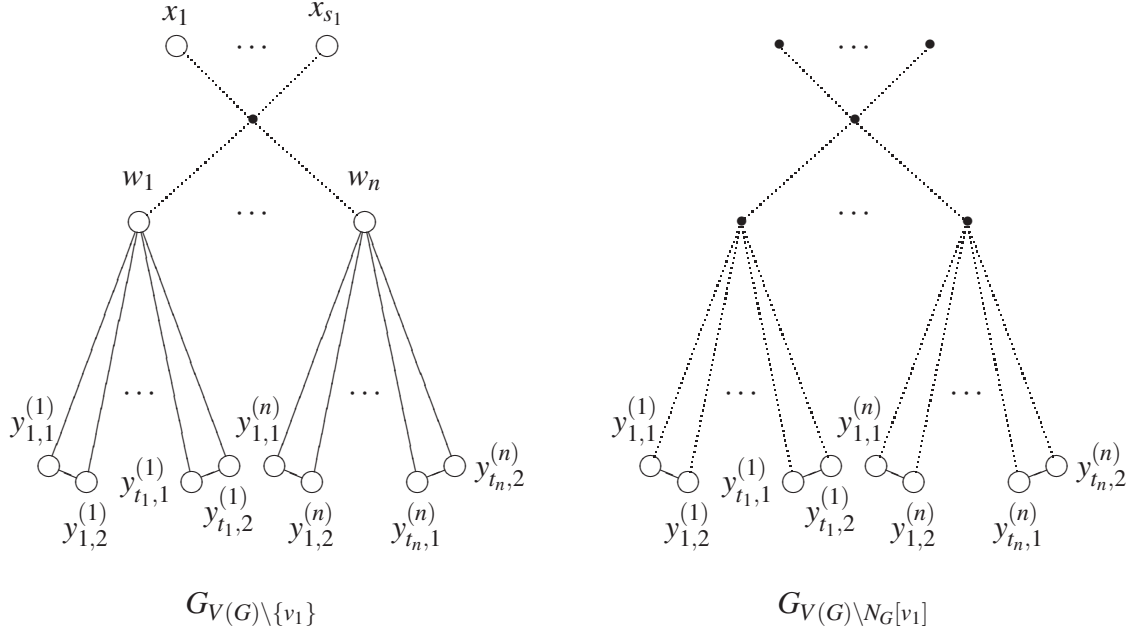

Hence, by using Lemmata \ref{Star}, \ref{StarTriangle} and \ref{HSofTensorProduct}, one has 

\[
H\left({\frac{K[V(G) \setminus \{v_{1}\}]}{I\left(G_{V(G) \setminus \{v_{1}\}}\right)},  \ \lambda}\right) = \frac{\prod_{j = 1}^{n} \left\{(1 + \lambda)^{t_{j}} + \lambda(1 - \lambda)^{t_{j} - 1} \right\}}{(1 - \lambda)^{s_{1} + \sum_{j = 1}^{n} t_{j}}}
\]
and 
\[
H\left({\frac{K[V(G) \setminus N_{G}[v_{1}]\ ]}{I\left(G_{V(G) \setminus N_{G}[v_{1}]}\right)}, \ \lambda}\right) = \frac{(1 + \lambda)^{\sum_{j = 1}^{n}t_{j}}}{(1 - \lambda)^{\sum_{j = 1}^{n}t_{j}}}. 
\]
%
%
Thus, by virtue of Lemma \ref{HSofEdgeIdeal}, it follows that  
\begin{eqnarray*}
& & H(K[V(G)]/I(G), \ \lambda) \\
&=& \frac{\prod_{j = 1}^{n} \left\{(1 + \lambda)^{t_{j}} + \lambda(1 - \lambda)^{t_{j} - 1} \right\}}{(1 - \lambda)^{s_{1} + \sum_{j = 1}^{n} t_{j}}}  +  \frac{(1 + \lambda)^{\sum_{j = 1}^{n}t_{j}}}{(1 - \lambda)^{\sum_{j = 1}^{n}t_{j}}} \cdot \frac{\lambda}{1 - \lambda} \\
&=& \frac{\prod_{j = 1}^{n} \left\{(1 + \lambda)^{t_{j}} + \lambda(1 - \lambda)^{t_{j} - 1} \right\} + \lambda(1 + \lambda)^{ \sum_{j = 1}^{n} t_{j} }(1 - \lambda)^{s_{1} - 1}}{(1 - \lambda)^{s_{1} + \sum_{j = 1}^{n} t_{j}}}. 
\end{eqnarray*}
Therefore one has 
\[
\deg h\left(K[V(G)]/I(G), \ \lambda \right) = \dim K[V(G)]/I(G) = s_{1} + \sum_{j = 1}^{n} t_{j}, 
\]
as desired.

\bigskip

{\bf (Second Step.)} Let $m > 1$ and $n = 1$. 
We will show 
\[
\deg h\left(K[V(G)]/I(G), \ \lambda \right) = \dim K[V(G)]/I(G) = \sum_{i = 1}^{m} s_{i} + \max\{t_{1}, 1\}.  
\]
Note that
\begin{itemize}
\item $G_{V(G) \setminus \{w_{1}\}}$ consists of $m + t_{1}$ star graphs $G^{{\rm star}(v_{1})}_{s_{1}}, \ldots, G^{{\rm star}(v_{m})}_{s_{m}}$ 
and $G^{{\rm star}(y^{(1)}_{1, 1})}_{1}, \ldots, G^{{\rm star}(y^{(1)}_{t_{1}, 1})}_{1}$, 
\item $G_{V(G) \setminus N_{G}[w_{1}]}$ consists of $\sum_{i = 1}^{m} s_{i}$ isolated vertices; 
\end{itemize}
see Figure \ref{fig:m>1;n=1}. 

\begin{figure}[htbp]
  \centering

\bigskip

\begin{xy}
	\ar@{} (0,0);(35,-45) *{\text{$G_{V(G) \setminus \{w_{1}\}}$}};	
	\ar@{} (0,0);(15, 0)  *\cir<4pt>{} = "E1";
	\ar@{-} "E1";(10, 8) *++!D{x^{(1)}_{1}} *\cir<4pt>{};
	\ar@{-} "E1";(20, 8) *++!D{x^{(1)}_{s_{1}}} *\cir<4pt>{};
	\ar@{} (0,0); (15, 11) *++!U{\cdots}
	\ar@{} (0,0);(55, 0)  *\cir<4pt>{} = "E";
	\ar@{-} "E";(50, 8) *++!D{x^{(m)}_{1}} *\cir<4pt>{};
	\ar@{-} "E";(60, 8) *++!D{x^{(m)}_{s_{m}}} *\cir<4pt>{};
	\ar@{} (0,0); (55, 11) *++!U{\cdots}
	\ar@{} (0,0);(15,-4) *{\text{$v_{1}$}};
	\ar@{} (0,0);(35, -4) *++!D{\cdots};
	\ar@{} (0,0);(55,-4) *{\text{$v_{m}$}};
	\ar@{} "E1";(35, -13.5)  *\dir<4pt>{*} = "F10";
	\ar@{.} "F10";(20, -30) *++!R{y^{(1)}_{1,1}} *\cir<4pt>{} = "F11";
	\ar@{.} "F10";(25, -32) *++!U{y^{(1)}_{1,2}} *\cir<4pt>{} = "F12";
	\ar@{-} "F11";"F12";
	\ar@{.} "F10";(45, -32) *++!R{y^{(1)}_{t_{1}, 1}} *\cir<4pt>{} = "Ft1";
	\ar@{.} "F10";(50, -30) *++!L{y^{(1)}_{t_{1}, 2}} *\cir<4pt>{} = "Ft2";
	\ar@{-} "Ft1";"Ft2";
	\ar@{} (0,0);(35,-28) *++!D{\cdots};
	\ar@{.} "E1";"F10";
	\ar@{.} "E";"F10";
	
	\ar@{} (0,0);(115,-45) *{\text{$G_{V(G) \setminus N_{G}[w_{1}]}$}};	
	\ar@{} (0,0);(95, 0)  *\dir<4pt>{*} = "E1r";
	\ar@{.} "E1r";(90, 8) *++!D{x^{(1)}_{1}} *\cir<4pt>{};
	\ar@{.} "E1r";(100, 8) *++!D{x^{(1)}_{s_{1}}} *\cir<4pt>{};
	\ar@{} (0,0); (95, 11) *++!U{\cdots}
	\ar@{} (0,0);(135, 0)  *\dir<4pt>{*} = "Er";
	\ar@{.} "Er";(130, 8) *++!D{x^{(m)}_{1}} *\cir<4pt>{};
	\ar@{.} "Er";(140, 8) *++!D{x^{(m)}_{s_{m}}} *\cir<4pt>{};
	\ar@{} (0,0); (135, 11) *++!U{\cdots}
	\ar@{} (0,0);(115, -4) *++!D{\cdots};
	\ar@{} "E1r";(115, -13.5)  *\dir<4pt>{*} = "F10r";
	\ar@{.} "F10r";(100, -30)  *\dir<4pt>{*} = "F11r";
	\ar@{.} "F10r";(105, -32)  *\dir<4pt>{*} = "F12r";
	\ar@{.} "F11r";"F12r";
	\ar@{.} "F10r";(125, -32)  *\dir<4pt>{*} = "Ft1r";
	\ar@{.} "F10r";(130, -30) *\dir<4pt>{*} = "Ft2r";
	\ar@{.} "Ft1r";"Ft2r";
	\ar@{} (0,0);(115,-28) *++!D{\cdots};
	\ar@{.} "E1r";"F10r";
	\ar@{.} "Er";"F10r";

\end{xy}

\bigskip

  \caption{$G_{V(G) \setminus \{w_{1}\}}$ (left) and $G_{V(G) \setminus N_{G}[w_{1}]}$ (right)}
  \label{fig:m>1;n=1}
\end{figure}

Hence, by using Lemmata \ref{Star} and \ref{HSofTensorProduct}, one has 

\[
H\left({\frac{K[V(G) \setminus \{w_{1}\}]}{I\left(G_{V(G) \setminus \{w_{1}\}}\right)},  \ \lambda}\right) = \frac{ \prod_{i = 1}^{m} \left\{ 1 + \lambda(1 - \lambda)^{s_{i} - 1} \right\} \cdot (1 + \lambda)^{t_{1}} }{ (1 - \lambda)^{\sum_{i = 1}^{m} s_{i} + t_{1}} }
\]
and 
\[
H\left({\frac{K[V(G) \setminus N_{G}[w_{1}]\ ]}{I\left(G_{V(G) \setminus N_{G}[w_{1}]}\right)}, \ \lambda}\right) = \frac{ 1 }{ (1 - \lambda)^{\sum_{i = 1}^{m} s_{i}} }. 
\]
Thus, by virtue of Lemma \ref{HSofEdgeIdeal}, it follows that  

\begin{eqnarray*}
& & H\left(K[V(G)]/I(G),\ \lambda \right) \\
&=& \frac{ \prod_{i = 1}^{m} \left\{ 1 + \lambda(1 - \lambda)^{s_{i} - 1} \right\} \cdot (1 + \lambda)^{t_{1}} }{ (1 - \lambda)^{\sum_{i = 1}^{m} s_{i} + t_{1}} }  +  \frac{ 1 }{ (1 - \lambda)^{\sum_{i = 1}^{m} s_{i}} } \cdot \frac{\lambda}{1 - \lambda} \\
&=& \frac{ \prod_{i = 1}^{m} \left\{ 1 + \lambda(1 - \lambda)^{s_{i} - 1} \right\} \cdot (1 + \lambda)^{t_{1}} }{ (1 - \lambda)^{\sum_{i = 1}^{m} s_{i} + t_{1}} }  +  \frac{ \lambda }{ (1 - \lambda)^{\sum_{i = 1}^{m} s_{i} + 1} } \\
&=& \frac{ \prod_{i = 1}^{m} \left\{ 1 + \lambda(1 - \lambda)^{s_{i} - 1} \right\} \cdot (1 + \lambda)^{t_{1}}(1 - \lambda)^{\max\{t_{1}, 1\} - t_{1}} + \lambda(1 - \lambda)^{\max\{t_{1}, 1\} - 1} }{ (1 - \lambda)^{\sum_{i = 1}^{m} s_{i} + \max\{t_{1}, 1\}} }.  
\end{eqnarray*}
Hence $\deg h\left(K[V(G)]/I(G), \ \lambda \right) = \sum_{i = 1}^{m} s_{i} + t_{1} + \max\{t_{1}, 1\} - t_{1} = \sum_{i = 1}^{m} s_{i} + \max\{t_{1}, 1\}$. 
Therefore, one has
\[
\deg h\left(K[V(G)]/I(G), \ \lambda \right) = \dim K[V(G)]/I(G) = \sum_{i = 1}^{m} s_{i} + \max\{t_{1}, 1\}, 
\]
as desired.

\bigskip

{\bf (Third Step.)} Let $m > 1$ and $n > 1$. 
Suppose that there exists $1 \leq \ell \leq n$ 
such that $\{v_m, w_{\ell}\}$ is a leaf edge. 
We may assume $\ell = n$. 
Then $t_{n} = 0$. 
We will show 
\[
\deg h\left(K[V(G)]/I(G), \ \lambda \right) = \dim K[V(G)]/I(G) = \sum_{i = 1}^{m} s_{i} + \sum_{j = 1}^{n - 1} \max\{ t_{j}, 1\} + 1. 
\]
Note that we can regard $G$ as a Cameron--Walker graph such that 
its bipartite part has bipartition 
$\{ v_1, \ldots, v_m \} \cup \{ w_1, \ldots, w_{n-1} \}$, 
the vertex $v_{i}$ has $s_{i}$ leaf edges for all $1 \leq i \leq m - 1$ and 
the vertex $v_{m}$ has $s_{m} + 1 $ leaf edges. 
Thus, by induction hypothesis, one has
\[
\deg h\left(K[V(G)]/I(G), \ \lambda \right) = \dim K[V(G)]/I(G) = \sum_{i = 1}^{m} s_{i} + 1 + \sum_{j = 1}^{n - 1} \max\{t_{j}, 1 \},  
\] 
as desired. 

\par
Next, suppose that $\{v_{m}, w_{\ell}\}$ is not a leaf edge 
for all $1\leq \ell \leq n$. 
Then $G_{V(G) \setminus \{v_{m}\}}$ consists of 
\begin{enumerate}
\item[(a1)] $s_{m}$ isolated vertices $x_{1}^{(m)}, \ldots, x_{s_m}^{(m)}$; 
\item[(a2)] star graphs $G^{{\rm star}(v_{i})}_{s_{i} + \alpha_{i}}$ for 
  $1 \leq i \leq m-1$ with 
  $N(v_i) \cap \{ w_1, \ldots, w_n \} 
     =: \{ w_{j_1}, \ldots, w_{j_{\alpha_i}} \}$ 
  satisfying $N(w_{j_k}) \subset \{ v_i, v_m \}$ for any 
  $k = 1, \ldots, \alpha_i$; 
\item[(a3)] star triangles $G^{\triangle (w_{j})}_{t_{j}}$ 
  for $1 \leq j \leq n$ 
  with $N(w_{j}) \cap \{v_{1}, \ldots, v_{m}\} = \{v_{m}\}$; 
\item[(a4)] some Cameron--Walker induced subgraphs. 
\end{enumerate}
We give an example after the proof; see Example \ref{ex:3Step}. 

\par
Note that each graph of type (a2) can be considered as 
a Cameron--Walker induced subgraph. 
Also note that each induced star graph $G^{{\rm star}(v_{i})}_{s_{i}}$ 
(resp.\  induced pendant triangle $G^{\triangle (w_{j})}_{t_{j}}$) 
appears in (a2) or (a4) (resp.\  (a3) or (a4)) as a (sub)graph. 
Hence by virtue of Lemmata \ref{Star}, \ref{StarTriangle}, \ref{disconnected} and induction hypothesis, one has  
\begin{eqnarray*}
\deg h\left({\frac{K[V(G) \setminus \{v_{m}\}]}{I\left(G_{V(G) \setminus \{v_{m}\}}\right)},  \ \lambda}\right) 
&\leq& \sum_{i = 1}^{m - 1} s_{i} + \sum_{j = 1}^{n} \max\left\{t_{j}, 1\right\}  
\end{eqnarray*}
and
\begin{eqnarray*}
\dim \frac{K[V(G) \setminus \{v_{m}\}]}{I\left(G_{V(G) \setminus \{v_{m}\}}\right)} 
&=& \sum_{i = 1}^{m - 1} s_{i} + \sum_{j = 1}^{n} \max\left\{t_{j}, 1\right\} + s_{m} \\
&=& \sum_{i = 1}^{m} s_{i} + \sum_{j = 1}^{n} \max\left\{t_{j}, 1\right\}. 
\end{eqnarray*}

On the other hand, $G_{V(G) \setminus N_{G}[v_{m}]}$ consists of 
\begin{enumerate}
\item[(b1)] star graphs $G^{{\rm star}(v_{i})}_{s_{i}}$ for $1 \leq i \leq m-1$ 
  with $N(v_i) \cap \{ w_1, \ldots, w_n \} \subset N(v_m)$;
\item[(b2)] star graphs $G^{{\rm star}(y^{(j)}_{\ell, 1})}_{1}$ 
  for $1 \leq j \leq n$ with $\{ v_{m}, w_{j} \} \in E(G)$ and $1 \leq \ell \leq t_{j}$. 
\item[(b3)] some Cameron--Walker induced subgraphs; 
\end{enumerate}
see Example \ref{ex:3Step}.

\par
Note that each induced star graph $G^{{\rm star}(v_{i})}_{s_{i}}$ 
appears in (b1) or (b3) as a (sub)graph. 
Also note that the star graphs 
$G^{{\rm star}(y^{(j)}_{\ell, 1})}_{1}$, $1 \leq \ell \leq t_j$ 
of type (b2) are the edges of the pendant triangle 
$G^{\triangle (w_{j})}_{t_{j}}$ and the total contributions of these graphs 
to the degree of $h$-polynomial and the dimension are both $t_j$. 
Hence, by virtue of Lemmata \ref{Star}, \ref{disconnected} and induction hypothesis, it follows that 

\begin{eqnarray*}
& & \deg h\left({\frac{K[V(G) \setminus N_{G}[v_{m}]\ ]}{I\left(G_{V(G) \setminus N_{G}[v_{m}]\ }\right)},  \ \lambda}\right) 
= \dim \frac{K[V(G) \setminus N_{G}[v_{m}]\ ]}{I\left(G_{V(G) \setminus N_{G}[v_{m}]\ } \right)} \\
&=& \sum_{i = 1}^{m - 1} s_{i} + \sum_{1 \leq j \leq n \atop \{v_{m}, w_{j}\} \not\in E(G) } \max\{ t_{j}, 1 \} + \sum_{1 \leq j \leq n \atop \{v_{m}, w_{j}\} \in E(G) } t_{j} \\
&<& \sum_{i = 1}^{m} s_{i} + \sum_{j = 1}^{n} \max\left\{t_{j}, 1\right\} = \dim \frac{K[V(G) \setminus \{v_{m}\}]}{I\left(G_{V(G) \setminus \{v_{m}\}}\right)}. 
\end{eqnarray*}
Thus Lemma \ref{Lemma} says that 
\[
\deg h\left(K[V(G)]/I(G), \lambda \right) = \dim K[V(G)]/I(G) = \sum_{i = 1}^{m}s_{i} + \sum_{j = 1}^{n} \max\left\{t_{j}, 1\right\}, 
\] 
as desired. 
\end{proof}

We give an example of Cameron--Walker graph with $m>1$ and $n>1$ 
which would be helpful to understand (Third Step.) of the proof of 
Proposition \ref{s=d}. 
\begin{Example}
  \label{ex:3Step}  
  Let $G$ be the following Cameron--Walker graph:  

\bigskip

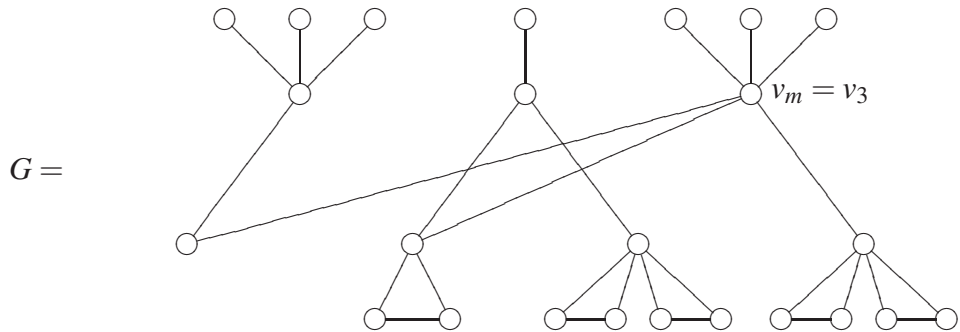
\begin{xy}
	\ar@{} (0,0);(15,-10) *{\text{$G = $}};
	\ar@{} (0,0);(50, 0)  *\cir<4pt>{} = "V1";
	\ar@{} (0,0);(80, 0)  *\cir<4pt>{} = "V2";
	\ar@{} (0,0);(110, 0) *++!L{v_{m} = v_{3}} *\cir<4pt>{} = "V3";
	\ar@{} (0,0);(35, -20)  *\cir<4pt>{} = "W1";
	\ar@{} (0,0);(65, -20)  *\cir<4pt>{} = "W2";
	\ar@{} (0,0);(95, -20)  *\cir<4pt>{} = "W3";
	\ar@{} (0,0);(125, -20)  *\cir<4pt>{} = "W4";
	\ar@{-} "V1";"W1";
	\ar@{-} "V2";"W2";
	\ar@{-} "V2";"W3";
	\ar@{-} "V3";"W1";
	\ar@{-} "V3";"W2";
	\ar@{-} "V3";"W4";
	\ar@{-} "V1";(50, 10) *\cir<4pt>{};
	\ar@{-} "V1";(40, 10) *\cir<4pt>{}; 
	\ar@{-} "V1";(60, 10) *\cir<4pt>{};
	\ar@{-} "V2";(80, 10) *\cir<4pt>{};
	\ar@{-} "V3";(110, 10) *\cir<4pt>{};
	\ar@{-} "V3";(100, 10) *\cir<4pt>{}; 
	\ar@{-} "V3";(120, 10) *\cir<4pt>{};
	\ar@{-} "W2";(60, -30) *\cir<4pt>{} = "T21";
	\ar@{-} "W2";(70, -30) *\cir<4pt>{} = "T22";
	\ar@{-} "T21";"T22";
	\ar@{-} "W3";(92, -30) *\cir<4pt>{} = "T31";
	\ar@{-} "W3";(84, -30) *\cir<4pt>{} = "T32";
	\ar@{-} "W3";(98, -30) *\cir<4pt>{} = "T33";
	\ar@{-} "W3";(106, -30) *\cir<4pt>{} = "T34";
	\ar@{-} "T31";"T32";
	\ar@{-} "T33";"T34";
	\ar@{-} "W4";(122, -30) *\cir<4pt>{} = "T41";
	\ar@{-} "W4";(114, -30) *\cir<4pt>{} = "T42";
	\ar@{-} "W4";(128, -30) *\cir<4pt>{} = "T43";
	\ar@{-} "W4";(136, -30) *\cir<4pt>{} = "T44";
	\ar@{-} "T41";"T42";
	\ar@{-} "T43";"T44";
\end{xy}

\bigskip

\par \noindent 
Then the induced subgraph $G_{V(G) \setminus \{v_{m}\}}$ is as follows. 

\bigskip

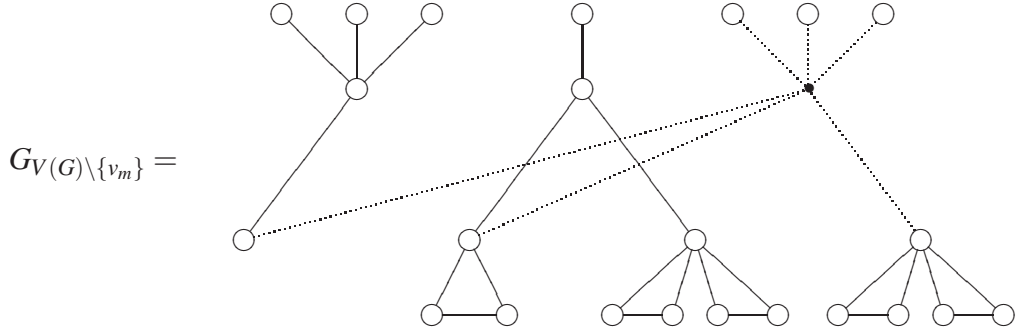
\begin{xy}
	\ar@{} (0,0);(15,-10) *{\text{$G_{V(G) \setminus \{v_{m}\}} = $}};
	\ar@{} (0,0);(50, 0)  *\cir<4pt>{} = "V1";
	\ar@{} (0,0);(80, 0)  *\cir<4pt>{} = "V2";
	\ar@{} (0,0);(110, 0) *\dir<4pt>{*} = "V3";
	\ar@{} (0,0);(35, -20)  *\cir<4pt>{} = "W1";
	\ar@{} (0,0);(65, -20)  *\cir<4pt>{} = "W2";
	\ar@{} (0,0);(95, -20)  *\cir<4pt>{} = "W3";
	\ar@{} (0,0);(125, -20)  *\cir<4pt>{} = "W4";
	\ar@{-} "V1";"W1";
	\ar@{-} "V2";"W2";
	\ar@{-} "V2";"W3";
	\ar@{.} "V3";"W1";
	\ar@{.} "V3";"W2";
	\ar@{.} "V3";"W4";
	\ar@{-} "V1";(50, 10) *\cir<4pt>{};
	\ar@{-} "V1";(40, 10) *\cir<4pt>{}; 
	\ar@{-} "V1";(60, 10) *\cir<4pt>{};
	\ar@{-} "V2";(80, 10) *\cir<4pt>{};
	\ar@{.} "V3";(110, 10) *\cir<4pt>{};
	\ar@{.} "V3";(100, 10) *\cir<4pt>{}; 
	\ar@{.} "V3";(120, 10) *\cir<4pt>{};
	\ar@{-} "W2";(60, -30) *\cir<4pt>{} = "T21";
	\ar@{-} "W2";(70, -30) *\cir<4pt>{} = "T22";
	\ar@{-} "T21";"T22";
	\ar@{-} "W3";(92, -30) *\cir<4pt>{} = "T31";
	\ar@{-} "W3";(84, -30) *\cir<4pt>{} = "T32";
	\ar@{-} "W3";(98, -30) *\cir<4pt>{} = "T33";
	\ar@{-} "W3";(106, -30) *\cir<4pt>{} = "T34";
	\ar@{-} "T31";"T32";
	\ar@{-} "T33";"T34";
	\ar@{-} "W4";(122, -30) *\cir<4pt>{} = "T41";
	\ar@{-} "W4";(114, -30) *\cir<4pt>{} = "T42";
	\ar@{-} "W4";(128, -30) *\cir<4pt>{} = "T43";
	\ar@{-} "W4";(136, -30) *\cir<4pt>{} = "T44";
	\ar@{-} "T41";"T42";
	\ar@{-} "T43";"T44";
\end{xy}

\bigskip

\par \noindent
Also the induced subgraph $G_{V(G) \setminus N_{G}[v_{m}]}$ is as follows. 

\bigskip

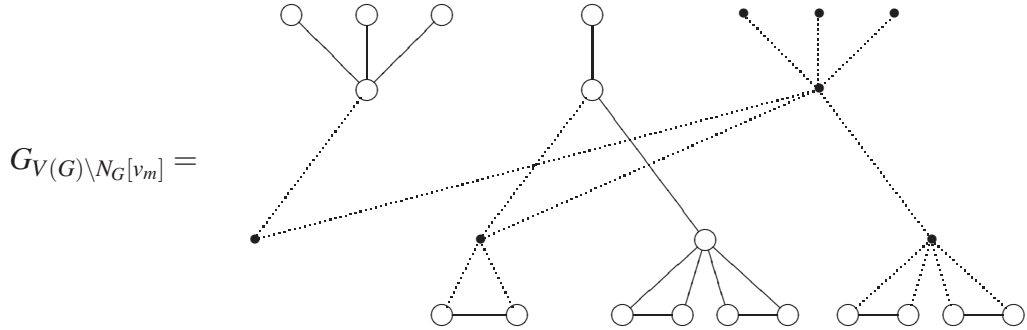
\begin{xy}
	\ar@{} (0,0);(15,-10) *{\text{$G_{V(G) \setminus N_{G}[v_{m}]} = $}};
	\ar@{} (0,0);(50, 0)  *\cir<4pt>{} = "V1";
	\ar@{} (0,0);(80, 0)  *\cir<4pt>{} = "V2";
	\ar@{} (0,0);(110, 0) *\dir<4pt>{*} = "V3";
	\ar@{} (0,0);(35, -20)  *\dir<4pt>{*} = "W1";
	\ar@{} (0,0);(65, -20)  *\dir<4pt>{*} = "W2";
	\ar@{} (0,0);(95, -20)  *\cir<4pt>{} = "W3";
	\ar@{} (0,0);(125, -20)  *\dir<4pt>{*} = "W4";
	\ar@{.} "V1";"W1";
	\ar@{.} "V2";"W2";
	\ar@{-} "V2";"W3";
	\ar@{.} "V3";"W1";
	\ar@{.} "V3";"W2";
	\ar@{.} "V3";"W4";
	\ar@{-} "V1";(50, 10) *\cir<4pt>{};
	\ar@{-} "V1";(40, 10) *\cir<4pt>{}; 
	\ar@{-} "V1";(60, 10) *\cir<4pt>{};
	\ar@{-} "V2";(80, 10) *\cir<4pt>{};
	\ar@{.} "V3";(110, 10) *\dir<4pt>{*};
	\ar@{.} "V3";(100, 10) *\dir<4pt>{*}; 
	\ar@{.} "V3";(120, 10) *\dir<4pt>{*};
	\ar@{.} "W2";(60, -30) *\cir<4pt>{} = "T21";
	\ar@{.} "W2";(70, -30) *\cir<4pt>{} = "T22";
	\ar@{-} "T21";"T22";
	\ar@{-} "W3";(92, -30) *\cir<4pt>{} = "T31";
	\ar@{-} "W3";(84, -30) *\cir<4pt>{} = "T32";
	\ar@{-} "W3";(98, -30) *\cir<4pt>{} = "T33";
	\ar@{-} "W3";(106, -30) *\cir<4pt>{} = "T34";
	\ar@{-} "T31";"T32";
	\ar@{-} "T33";"T34";
	\ar@{.} "W4";(122, -30) *\cir<4pt>{} = "T41";
	\ar@{.} "W4";(114, -30) *\cir<4pt>{} = "T42";
	\ar@{.} "W4";(128, -30) *\cir<4pt>{} = "T43";
	\ar@{.} "W4";(136, -30) *\cir<4pt>{} = "T44";
	\ar@{-} "T41";"T42";
	\ar@{-} "T43";"T44";
\end{xy}

\bigskip

\end{Example}


\section{Cameron--Walker graphs with the equality $(*)$}
\label{sec:(ast)}
As noted in Introduction, for an arbitrary finite simple graph $G$, 
one has 
\[
   \deg h\left(S/I(G), \ \lambda \right) - \reg\left( S/I(G) \right) 
       \leq \dim S/I(G) - \depth\left( S/I(G) \right), 
\]
where we set $S = K[V(G)]$. 
Then it is natural to ask for which graph $G$ satisfies the equality: 
\[
(*) \ \deg h\left(S/I(G), \ \lambda \right) - \reg\left( S/I(G) \right) 
       = \dim S/I(G) - \depth\left( S/I(G) \right). 
\]
Recall that the equality $(\ast)$ holds if and only if 
$S/I(G)$ has a unique extremal Betti number. 
Hence when $I(G)$ has a pure resolution 
(\cite[p.\  153]{BH}), the equality $(\ast)$ holds. 
Moreover by (\cite[Lemma 3]{BiHe}), 
it follows that the equality $(\ast)$ holds 
if $S/I(G)$ is Cohen--Macaulay. 

\par
In this section, we give a classification of Cameron--Walker graphs $G$ 
with the equality $(*)$. 

\par
Throughout this section, let $G$ be a Cameron--Walker graph 
whose labeling of vertices is as in Figure \ref{fig:CameronWalkerGraph}. 
By Theorem \ref{CW}, the equality $(*)$ holds if and only if 
$\displaystyle \depth\left(S/I(G)\right) = \reg\left(S/I(G)\right)$. 
Both of these invariants have combinatorial explanations. 
The regularity is equal to the induced matching number 
(or the matching number) of $G$: 
$\reg\left(S/I(G)\right) = \sum_{j = 1}^{n}t_{j} + m$. 
In order to state about the depth, we need some definitions. 

\par
For a subset $A \subset V(G)$, 
we set $N_{G}(A) = \bigcup_{v \in A} N_{G}(v) \setminus A$. 
A subset $A \subset V(G)$ is said to be {\em independent} if 
$\{ x_i, x_j \} \notin E(G)$ for any $x_i, x_j \in A$. 
We denote by $i(G)$, the minimum cardinality of independent sets $A$ 
with $A \cup N_{G}(A) = V(G)$. 
Then $\displaystyle \depth\left(S/I(G)\right) = i(G)$; 
see \cite[Corollary 3.7]{HHKO}. 

\par
We have the following estimation for $i(G)$. 
\begin{Lemma}\label{i(G)}
  Let $G$ be a Cameron--Walker graph 
  whose labeling of vertices is as in Figure \ref{fig:CameronWalkerGraph}. 
  Then
\[
m + |\{j : t_{j} > 0\}| \leq i(G) \leq \min\left\{\sum_{i = 1}^{m}s_{i} + n, \sum_{j = 1}^{n}t_{j} + m\right\}. 
\]

Moreover if the bipartite part of $G$ is the complete bipartite graph, 
then 
\[
  i(G) = \min\left\{\sum_{i = 1}^{m}s_{i} + n, \sum_{j = 1}^{n}t_{j} + m\right\}. 
\]
\end{Lemma}
\begin{proof}
The upper bound is clear. We prove the lower bound. 

\par 
Let $A \subset V(G)$ be an independent set with $A \cup N_{G}(A) = V(G)$. 
Then we put 
$A_{\rm bip} = A \cap \{v_{1}, \ldots, v_{m}, w_{1}, \ldots, w_{n}\}$ and 
$A' = A \setminus A_{\rm bip}$. 
We note that 
$A = A_{\rm bip} \sqcup A'$, and 
\begin{itemize}
\item If $v_{i} \not\in A_{\rm bip}$, then $x^{(i)}_{1}, \ldots, x^{(i)}_{s_{i}} \in A'$;
\item If $w_{j} \not\in A_{\rm bip}$, then $y^{(j)}_{\ell, 1} \in A'$ 
  or $y^{(j)}_{\ell, 2} \in A'$ for all $1 \leq \ell \leq t_{j}$. 
\end{itemize}
Hence one has 
\begin{eqnarray*}
|A| = |A_{\rm bip}| + |A'| &\geq& |A_{\rm bip}| + \sum_{1 \leq i \leq m \atop v_{i} \not\in A_{\rm bip}} s_{i} 
+ \sum_{1 \leq j \leq n \atop w_{j} \not\in A_{\rm bip}} t_{j} \\
&\geq& m + |\{j : t_{j} > 0\}|. 
\end{eqnarray*}
Thus $i(G) \geq m + |\{j : t_{j} > 0\}|$. 

\par
When the bipartite part of $G$ is the complete bipartite graph, 
one has either $A_{\rm bip} \subset \{ v_1, \ldots, v_m \}$ or 
$A_{\rm bip} \subset \{ w_1, \ldots, w_n \}$. 
For the former case, since $s_i \geq 1$ for all $i$, it follows that 
$|A| \geq \sum_{j=1}^n t_j + m$. 
For the latter case, one has 
$|A| \geq \sum_{i=1}^m s_i + n$ 
because $w_j \in A_{\rm bip}$ if $t_j = 0$. 
It then follows that 
\[
  i(G) \geq \min\left\{\sum_{i = 1}^{m}s_{i} + n, \sum_{j = 1}^{n}t_{j} + m\right\}. 
\]
Combining this with the upper bound, one has the equality. 
\end{proof}

%

By virtue of this lemma, we can give a classification of Cameron--Walker graphs $G$ 
satisfying the equality $(*)$. 

\begin{Theorem}\label{Main}
  Let $G$ be a Cameron--Walker graph 
  whose labeling of vertices is as in Figure \ref{fig:CameronWalkerGraph} and 
  $G_{\rm bip}$ the bipartite part of $G$. 
  Then $S/I(G)$ satisfies the equality $(*)$ if and only if 
  \begin{eqnarray}\label{ineqThm}
  \sum_{1 \leq i \leq m \atop v_{i} \in V} s_{i} \ + \ \left| \left\{ j : N_{G_{\rm bip}} (w_{j}) \subset V  \right\} \right| \geq \sum_{1 \leq j \leq n \atop N_{G_{\rm bip}} (w_{j}) \subset V} t_{j} \ + \ |V|
  \end{eqnarray}
  holds for all $V \subset \{ v_{1}, \ldots, v_{m} \}$. 
\end{Theorem}
\begin{proof}
Assume that there exists a subset $V \subset \{ v_{1}, \ldots, v_{m} \}$ satisfying
  \[
  \sum_{1 \leq i \leq m \atop v_{i} \in V} s_{i} \ + \ \left| \left\{ j : N_{G_{\rm bip}} (w_{j}) \subset V  \right\} \right| < \sum_{1 \leq j \leq n \atop N_{G_{\rm bip}} (w_{j}) \subset V} t_{j} \ + \ |V|. 
  \]
Let
\[
A = \left( \{ v_{1}, \ldots, v_{m} \} \setminus V \right) \ \cup \ \left\{ w_{j} : N_{G_{\rm bip}} (w_{j}) \subset V \right\} \ \ \ \ \ \ \ \ \ \ \ \ \ \ \ \ 
\]
\[ 
\cup \bigcup_{1 \leq i \leq m \atop v_{i} \in V} \left\{ x_{1}^{(i)}, \ldots, x_{s_{i}}^{(i)} \right\} \ \cup \ \bigcup_{1 \leq j \leq n \atop N_{G_{\rm bip}} (w_{j}) \not\subset V, \ t_{j} > 0} \left\{ y_{1, 1}^{(j)}, \ldots, y_{t_{j}, 1}^{(j)} \right\}. 
\] 
Then $A$ is an independent set with $A \cup N_{G} (A) = V(G)$ and 
\begin{eqnarray*}
|A| &=& m - |V| + \left| \left\{ j : N_{G_{\rm bip}} (w_{j}) \subset V \right\} \right| + \sum_{1 \leq i \leq m \atop v_{i} \in V} s_{i} + \sum_{1 \leq j \leq n \atop N_{G_{\rm bip}} (w_{j}) \not\subset V, \ t_{j} > 0} t_{j} \\
&<& m - |V| + \sum_{1 \leq j \leq n \atop N_{G_{\rm bip}} (w_{j}) \subset V} t_{j} \ + \ |V| + \sum_{1 \leq j \leq n \atop N_{G_{\rm bip}} (w_{j}) \not\subset V, \ t_{j} > 0} t_{j} \\ 
&=& \sum_{j = 1}^{n} t_{j} + m. 
\end{eqnarray*}
Hence we have 
\[
\depth(S/I(G)) = i(G) < \sum_{j = 1}^{n} t_{j} + m = \reg(S/I(G)). 
\]
Thus $S/I(G)$ does not satisfy the equality $(*)$. 

Next, we assume that 
  \[
  \sum_{1 \leq i \leq m \atop v_{i} \in V} s_{i} \ + \ \left| \left\{ j : N_{G_{\rm bip}} (w_{j}) \subset V  \right\} \right| \geq \sum_{1 \leq j \leq n \atop N_{G_{\rm bip}} (w_{j}) \subset V} t_{j} \ + \ |V|
  \]
  holds for all $V \subset \{ v_{1}, \ldots, v_{m} \}$. 

Let $A$ be an independent set of $V(G)$ with $A \cup N_{G} (A) = V(G)$.  
Let $A_{v} = A \cap \{ v_{1}, \ldots, v_{m} \}$ and $A_{w} = A \cap \{ w_{1}, \ldots, w_{n} \}$. 
Then, 
\[
|A| = |A_{v}| + |A_{w}| + \sum_{1 \leq i \leq m \atop v_{i} \in \{ v_{1}, \ldots, v_{m} \} \setminus A_{v}} s_{i} + \sum_{1 \leq j \leq n \atop N_{G_{\rm bip}} (w_{j}) \not\subset \{ v_{1}, \ldots, v_{m} \} \setminus A_{v}} t_{j} + \sum_{1 \leq j \leq n \atop N_{G_{\rm bip}} (w_{j}) \subset \{ v_{1}, \ldots, v_{m} \} \setminus A_{v}, w_{j} \not\in A_{w}} t_{j}. 
\]
For $j$ satisfying $N_{G_{\rm bip}} (w_j) \subset \{ v_1, \ldots, v_m \} \setminus A_{v}$ and $w_j \not\in A_w$, one has $t_{j} \geq 1$. 
Hence
\begin{eqnarray*}
|A| &\geq& |A_{v}| + |A_{w}| + \sum_{1 \leq i \leq m \atop v_{i} \in \{ v_{1}, \ldots, v_{m} \} \setminus A_{v}} s_{i} + \sum_{1 \leq j \leq n \atop N_{G_{\rm bip}} (w_{j}) \not\subset \{ v_{1}, \ldots, v_{m} \} \setminus A_{v}} t_{j} \\ 
& & + \left| \left\{ j : N_{G_{\rm bip}} (w_{j}) \subset \{ v_{1}, \ldots, v_{m} \} \setminus A_{v} , w_{j} \not\in A_{w} \right\} \right|. 
\end{eqnarray*}
Since $N_{G_{\rm bip}} (w_{j}) \subset \{ v_{1}, \ldots, v_{m} \} \setminus A_{v}$ 
if $w_{j} \in A_{w}$, one has
\[
|A| \geq |A_{v}| +  \sum_{1 \leq i \leq m \atop v_{i} \in \{ v_{1}, \ldots, v_{m} \} \setminus A_{v}} s_{i} + \left| \left\{ j : N_{G_{\rm bip}} (w_{j}) \subset \{ v_{1}, \ldots, v_{m} \} \setminus A_{v} \right\} \right| + \sum_{1 \leq j \leq n \atop N_{G_{\rm bip}} (w_{j}) \not\subset \{ v_{1}, \ldots, v_{m} \} \setminus A_{v}} t_{j}. 
\]
Considering the inequality (\ref{ineqThm}) for $V = \{ v_{1}, \ldots, v_{m} \} \setminus A_{v}$, 
it follows that 
\begin{eqnarray*}
& & \sum_{1 \leq i \leq m \atop v_{i} \in \{v_{1}, \ldots, v_{m}\} \setminus A_{v}} s_{i} + \left| \left\{ j : N_{G_{\rm bip}} (w_{j}) \subset \{ v_{1}, \ldots, v_{m} \} \setminus A_{v} \right\} \right| \\
&\geq& \sum_{1 \leq j \leq n \atop N_{G_{\rm bip}} (w_{j}) \subset \{ v_{1}, \ldots, v_{m} \} \setminus A_{v}}t_{j} \ + \  \left| \{ v_{1}, \ldots, v_{m} \} \setminus A_{v} \right|. 
\end{eqnarray*}
Hence we have 
\begin{eqnarray*}
|A| &\geq& |A_{v}| + \sum_{1 \leq j \leq n \atop N_{G_{\rm bip}} (w_{j}) \subset \{ v_{1}, \ldots, v_{m} \} \setminus A_{v}}t_{j} \ + \  \left| \{ v_{1}, \ldots, v_{m} \} \setminus A_{v} \right| + \sum_{1 \leq j \leq n \atop N_{G_{\rm bip}} (w_{j}) \not\subset \{ v_{1}, \ldots, v_{m} \} \setminus A_{v}} t_{j} \\
&=& \sum_{j = 1}^{n} t_{j} + m. 
\end{eqnarray*}
Thus one has 
\[
i(G) \geq \sum_{j = 1}^{n} t_{j} + m.
\] 
This inequality together with Lemma \ref{i(G)} says that  
\[
\depth(S/I(G)) = i(G) = \sum_{j = 1}^{n} t_{j} + m = \reg(S/I(G)).
\] 
Therefore $S/I(G)$ satisfies the equality $(*)$. 
\end{proof}

\begin{Remark}\normalfont
\label{remark}
\begin{enumerate}
	\item When we use Theorem \ref{Main}, we only need to check the inequality (\ref{ineqThm}) for $V \subset \{ v_{1}, \ldots, v_{m} \}$ with $N_{G_{\rm bip}} (w_{j}) \subset V$ for some $1 \leq j \leq n$.
	Indeed, let $V$ be a subset of $\{v_{1}, \ldots, v_{m}\}$ such that $N_{G_{\rm bip}} (w_{j}) \not\subset V$ for all $1 \leq j \leq n$. 
	Then the inequality (\ref{ineqThm}) for $V$ is 
	$\sum_{1 \leq i \leq m , v_{i} \in V} s_{i} \geq |V|$, 
	which always holds since $s_{i} \geq 1$ for all $1 \leq i \leq m$. 
	\item Considering the inequality (\ref{ineqThm}) for $V = \{ v_{1}, \ldots, v_{m} \}$, 
	it follows that $\sum_{i = 1}^{m}s_{i} + n \geq \sum_{j = 1}^{n}t_{j} + m$  
	holds if $S/I(G)$ satisfies the equality $(*)$. 
\end{enumerate}
\end{Remark}


As a corollary of Theorem \ref{Main}, one has

\begin{Corollary}\label{main}
  Let $G$ be a Cameron--Walker graph 
  whose labeling of vertices is as in Figure \ref{fig:CameronWalkerGraph}. 
  Suppose that $t_{j} \leq 1$ for all $1 \leq j \leq n$. 
  Then $S/I(G)$ satisfies the equality $(*)$. 
\end{Corollary}
\begin{Remark}\normalfont
  Let $G$ be a Cameron--Walker graph whose labeling of vertices is 
  as in Figure \ref{fig:CameronWalkerGraph}. 
  Then $S/I(G)$ is Cohen-Macaulay if and only if 
  $s_{i} = 1$ for all $1 \leq i \leq m$ 
  and $t_j = 1$ for all $1 \leq j \leq n$ (\cite[Theorem 1.3]{HHKO}). 
  Hence the class of graphs in Corollary \ref{main} contains all Cohen--Macaulay 
  Cameron--Walker graphs. 
\end{Remark}
\begin{proof}[Proof of Corollary \ref{main}]
%
  Since $s_{i} \geq 1$ for all $1 \leq i \leq m$ and $t_j \leq 1$ for all $1 \leq j \leq n$, one has 
  \[
  \sum_{1 \leq i \leq m \atop v_{i} \in V} s_{i} \geq |V| \ \ \ {\rm and} \ \ \  \left| \left\{ j : N_{G_{\rm bip}} (w_{j}) \subset V  \right\} \right| \geq \sum_{1 \leq j \leq n \atop N_{G_{\rm bip}} (w_{j}) \subset V} t_{j}
  \]
   for all $V \subset \{v_{1}, \ldots, v_{m}\}$. 
   Hence $S/I(G)$ satisfies the equality $(*)$ by Theorem \ref{Main}. 
\end{proof}

From Theorem \ref{Main}, we also have 

\begin{Corollary}
  \label{Complete}
  Let $G$ be a Cameron--Walker graph 
  whose bipartite part is the complete bipartite graph. 
  We label the vertices of $G$ as in Figure \ref{fig:CameronWalkerGraph}. 
  Then $S/I(G)$ satisfies the equality $(*)$ if and only if 
  $\sum_{i = 1}^{m}s_{i} + n \geq \sum_{j = 1}^{n}t_{j} + m$.  
\end{Corollary}
\begin{proof} 
Since $N_{G_{\rm bip}}(w_{j}) = \{v_{1}, \ldots, v_{m}\}$ for all $1 \leq j \leq n$, the claim follows from Theorem \ref{Main} and Remark \ref{remark}. 
\end{proof}

\par
\bigskip

\par
In general, one has $\dim S/I(G) \geq \depth (S/I(G))$. 
Then it is natural to ask the following 
\begin{Question}\normalfont
  \label{quest:de}
  Given arbitrary integers $d, e$ with $d \geq e \geq 1$, 
  are there a Cameron--Walker graph $G$ 
  satisfying $\dim S/I(G) = d$ and $\depth (S/I(G)) = e$? 
\end{Question}

As an application of Corollary \ref{main}, 
we give a complete answer for Question \ref{quest:de}. 

\par
We first note about the depth. 
\begin{Proposition}
  \label{depth2}
  Let $G$ be a Cameron--Walker graph. Then $\depth S/I(G) \geq 2$. 
  Moreover $\depth S/I(G) = 2$ if and only if $G$ can be considered as 
  one of the following Cameron--Walker graphs: 
  \begin{enumerate}
  \item[(e1)] $m = 2$ and $t_{j} = 0$ for all $1 \leq j \leq n ;$ 
  \item[(e2)] $m = n = 1$ and $t_{1} = 1;$
  \item[(e3)] $m = n = 1$, $t_{1} \geq 2$, and $s_{1} = 1.$ 
  \end{enumerate}
  Here, we use labeling of vertices of $G$ 
  as in Figure \ref{fig:CameronWalkerGraph}. 
\end{Proposition}
\begin{proof}
  Assume that $G$ is a Cameron--Walker graph with $\depth(S/I(G)) = 1$. 
  By Lemma \ref{i(G)}, one has $m = 1$ and $t_{j} = 0$ for all 
  $1 \leq j \leq n$. 
  Then $G$ is a star graph but this is a contradiction 
  since star graphs are not Cameron--Walker by definition. 

  \par
  Next assume that $G$ is a Cameron--Walker graph with $\depth(S/I(G)) = 2$. 
  By Lemma \ref{i(G)}, one has 
  \begin{itemize}
  \item $m=2$ and $t_j =0$ for all $1 \leq j \leq n$, or
  \item $m=1$ and $t_j = 0$ except for one $j$. 
  \end{itemize}
  We consider the case $m=1$. Since $G$ is not a star graph, there exists 
  just one $j$ with $t_j \neq 0$, say $j=1$. 
  When $n \geq 2$, since $m=1$ and $t_j = 0$ for $2 \leq j \leq n$, 
  $G$ can be considered as a Cameron--Walker graph whose bipartite subgraph is 
  of type $(1,1)$ such that $v_1$ has $s_1 + (n-1)$ leaf edges and $w_1$ has 
  one pendant triangle. Thus we may assume $n=1$. 
  If $t_1 \geq 2$, then $i(G) = \depth S/I(G) =2$ implies that $s_1 = 1$. 
  Hence the assertion follows. 

  \par
  The converse is easy. 
\end{proof}

Since any Cameron--Walker graph $G$ satisfies $\depth S/I(G) \geq 2$, 
we only consider the case $e \geq 2$ in Question \ref{quest:de}. 
By virtue of Corollary \ref{main}, we can give a Cameron--Walker graph 
$G$ satisfying the properties in Question \ref{quest:de} with the equality 
$(\ast)$.  
\begin{Corollary}
  \label{de}
  Given arbitrary integers $d, e$ with $d \geq e \geq 2$, 
  there exists a Cameron--Walker graph $G$ with the equality $(\ast)$ 
  satisfying $\dim S/I(G) = d$ and $\depth(S/I(G)) = e$. 
\end{Corollary}
\begin{proof}
  We use the labeling of vertices of a Cameron--Walker graph 
  as in Figure \ref{fig:CameronWalkerGraph}. 
  \newline
  {\bf $\bullet$ The case $d>e$:} 
  Let $G$ be the Cameron--Walker graph with $m=e$, $n=1$, 
  $s_1 = \cdots = s_{e-1} = 1$, $s_{e} = d-e$, and $t_1 = 0$. 
  Then $\dim (S/I(G)) = \sum_{i=1}^e s_i + \max \{ t_1, 1 \} = d$. 
  Also, $A := \{ v_1, \ldots, v_e \}$ 
  is an independent set of $V(G)$ with $A \cup N_G (A) = V(G)$ 
  which gives $i(G)$. Thus one has $\depth S/I(G) = i(G) = |A| = e$. 
  \newline
  {\bf $\bullet$ The case $d=e$:} 
  Let $G$ be the Cameron--Walker graph with $m=d-1$, $n=1$, 
  $s_1 = \cdots = s_{d-1} = 1$, and $t_1 = 1$. 
  Then $\dim (S/I(G)) = \sum_{i=1}^{d-1} s_i + \max \{ t_1, 1 \} = d$. 
  Also, $A := \{ x_1^{(1)}, \ldots, x_{d-1}^{(1)} \} \cup \{ w_1 \}$ 
  is an independent set of $V(G)$ with $A \cup N_G (A) = V(G)$ 
  which gives $i(G)$. Thus one has $\depth S/I(G) = i(G) = |A| = e$. 
\end{proof}

\par
\bigskip

\par
Finally of the section, we provide some classes of graphs $G$ 
which satisfy the equality $(\ast)$ 
other than Cameron--Walker graphs. 


\begin{Proposition}
  \label{other-ast}
  Let $G$ be the one of the following graph. 
  Then the equality $(\ast)$ satisfies: 
  \begin{enumerate}
  \item The star graph $G^{{\rm star}(x_{v})}_{s}$ $(s \geq 1)$. 
  \item The path graph $P_n$ $(n \geq 2)$. 
  \item The $n$-cycle $C_n$ $(n \geq 3)$. 
  \item The graph $G_s$ on $\{ x_1, \ldots, x_{s+4} \}$ where $s \geq 1$ 
    which consists of the star graph $G^{{\rm star}(x_{s+3})}_{s}$ 
    on $\{ x_1, \ldots, x_s \} \cup \{ x_{s+3} \}$ 
    and $P_4$ on $\{ x_{s+1}, \ldots,b x_{s+4} \}$; see Figure \ref{fig:Gs}. 
    \begin{figure}[htbp]
      \centering

      \bigskip

      \begin{xy}
	\ar@{} (0,0);(30, 0) *{\text{$G_{s} =$}};
	\ar@{} (0,0);(50, 0) *++!U{x_{s + 1}} *\cir<4pt>{} = "A";
	\ar@{-} "A";(70, 0) *++!U{x_{s + 2}} *\cir<4pt>{} = "B";
	\ar@{-} "B";(90, 0) *++!U{x_{s + 3}} *\cir<4pt>{} = "C";
	\ar@{-} "C";(110, 0) *++!U{x_{s + 4}} *\cir<4pt>{} = "D";
	\ar@{-} "C";(80, 10) *++!D{x_{1}} *\cir<4pt>{};
	\ar@{-} "C";(100, 10) *++!D{x_{s}} *\cir<4pt>{};
	\ar@{} (0,0); (90, 14) *++!U{\cdots};
      \end{xy}

      \bigskip

      \caption{The graph $G_s$}
      \label{fig:Gs}
    \end{figure}
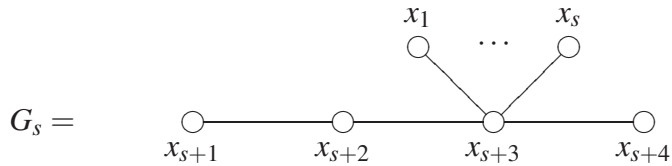
  \end{enumerate}
\end{Proposition}

Before proving Proposition \ref{other-ast}, 
we recall some facts on invariants of an edge ideal. 
%
For a finite simple graph $G$, 
the dimension $\dim S/I(G)$ is equal to the maximum cardinality 
of independent sets of $G$. 
In particular, one has 
$\dim S/I(P_{n}) = \lceil n/2 \rceil$ and 
$\dim S/I(C_{n}) = \lceil (n-1)/2 \rceil$. 


\par
We also recall the non-vanishing theorem of Betti numbers of edge ideals. 
%
\begin{Lemma}[{\cite[Theorems 3.1 and 4.1]{Kimura12}}]
  \label{nonvanishing}
  Let $G$ be a finite simple graph. 
  Suppose that there exists a set of star subgraphs 
  $\{ B_{1}, \ldots, B_{\ell} \} \ (\ell \geq 1)$ of $G$ 
  satisfying the following conditions$:$ 
  \begin{enumerate}
  \item $V(B_{k}) \cap V(B_{{k}'}) = \emptyset$ 
  for all $1 \leq k < {k}' \leq \ell; $
  \item There exist edges $e_{1}, \ldots, e_{\ell}$ with $e_{k} \in E(B_{k})$, 
    $k = 1, \ldots, \ell$ such that $\{e_{1}, \ldots, e_{\ell}\}$ 
    forms an induced matching of $G$. 
  \end{enumerate}
  Set $B_k 
  = G^{{\rm star}(x_{\beta_k})}_{\alpha_k}$ $(1 \leq k \leq \ell)$ and 
  $i = \alpha_{1} + \cdots + \alpha_{\ell}$. Then one has 
  \[
    \beta_{i, i+\ell} (S/I(G)) \neq 0. 
  \] 

  \par
  Moreover, when $G$ has no cycle, $\beta_{i,i+\ell} (S/I(G)) \neq 0$ 
  if and only if there exists such a set of star subgraphs of $G$. 
\end{Lemma}
By Lemma \ref{nonvanishing}, it follows that the equality 
$\reg (S/I(G)) = im(G)$ holds when $G$ has no cycle, 
which was first proved by Zheng \cite{Z}. 

\par
Now we prove Proposition \ref{other-ast}. 

\begin{proof}[Proof of Proposition \ref{other-ast}]
  Recall that the equality $(\ast)$ is satisfied if and only if 
  $(p, p+r)$-th Betti number does not vanish where 
  $p$ is the projective dimension and $r$ is the regularity. 
  \begin{enumerate}
  \item 
    Since $G^{{\rm star}(x_{v})}_{s}$ has no cycle, 
    one has $\reg (S/I(G^{{\rm star}(x_{v})}_{s})) = im (G) = 1$ by \cite{Z}. 
    Also, it is easy to see from Lemma \ref{nonvanishing} that 
    $\projdim (S/I(G^{{\rm star}(x_{v})}_{s})) = s$, 
    and $\beta_{s, s+1} (S/I(G^{{\rm star}(x_{v})}_{s})) \neq 0$. 
  \item 
    Let $V(P_{n}) = \{ x_{1}, x_{2}, \ldots x_{n} \}$ 
    and $E(P_{n}) = \left\{ \{x_{1}, x_{2}\}, \{x_{2}, x_{3}\}, \ldots, 
     \{x_{n-1}, x_{n}\} \right\}$. 
    It follows from \cite[Lemma 2.8]{Morey} that 
    $\depth(S/I(P_{n})) = \lceil n/3 \rceil$. 
    Hence by Auslander--Buchsbaum Theorem, one has 
    \begin{displaymath}
      p := \projdim(S/I(P_{n})) = n - \depth(S/I(P_{n})) 
         = n - \lceil n/3 \rceil. 
    \end{displaymath}
    Also, by \cite[p.4, Proposition]{Bouchat}, 
    one has 
    \begin{displaymath}
      r := \reg(S/I(P_{n})) = \lceil (n-1)/3 \rceil. 
    \end{displaymath}
    \begin{itemize}
    \item {\bf The case $n = 3\ell$ or $n = 3\ell + 1$ : } 
      Then $p = 2\ell$ and $r = \ell$. 
      For $1 \leq k \leq \ell$, let $B_{k}$ be the induced subgraph of $P_n$ 
      on $\{x_{3(k-1) + 1}, x_{3(k-1) + 2}, x_{3k}\}$. 
      Then $B_k$ is the star subgraph 
      $G^{{\rm star}(x_{3(k-1)+2})}_{2}$. 
      Take $e_{k} := \{ x_{3(k-1) + 1}, x_{3(k-1) + 2} \} \in E(B_{k})$. 
      Then $\{ e_{1}, \ldots, e_{\ell} \}$ forms an induced matching of $P_n$. 
      Thus Lemma \ref{nonvanishing} says that 
      $\beta_{p, p + r}(S/I(P_{n})) = \beta_{2\ell, 2\ell + \ell}(S/I(P_{n})) 
      \neq 0$. 
    \item {\bf The case $n = 3\ell + 2$ : } 
      Then $p = 2\ell + 1$ and $r = \ell + 1$. 
      For $1 \leq k \leq \ell$, let $B_{k}$ be the induced subgraph of $P_n$ 
      on $\{x_{3(k-1) + 1}, x_{3(k-1) + 2}, x_{3k}\}$. 
      Then $B_k$ is the star subgraph 
      $G^{{\rm star}(x_{3(k-1)+2})}_{2}$. 
      Also let $B_{\ell + 1}$ be the induced subgraph of $P_n$ 
      on $\{x_{3\ell+1}, x_{3\ell+2}\}$, 
      which is the star subgraph 
      $G^{{\rm star}(x_{3\ell +2})}_{1}$. 
      Take $e_{k} := \{ x_{3(k-1) + 1}, x_{3(k-1) + 2} \} \in E(B_{k})$ 
      for $k= 1,\ldots, \ell, \ell + 1$. 
      Then $\{ e_{1}, \ldots, e_{\ell}, e_{\ell + 1} \}$ 
      forms an induced matching of $P_n$. 
      Thus Lemma \ref{nonvanishing} says that 
      $\beta_{p, p + r}(S/I(P_{n})) 
      = \beta_{2\ell+1, (2\ell+1) + \ell+1}(S/I(P_{n})) \neq 0$.
    \end{itemize}
  \item Let $V(C_{n}) = \{ x_{1}, x_{2}, \ldots x_{n} \}$ 
    and $E(C_{n}) = \left\{ \{x_{1}, x_{2}\}, \ldots, 
     \{x_{n-1}, x_{n}\}, \{x_{1}, x_{n}\}  \right\}$. 
    It follows from \cite[p.\  117]{Cimpoeas} that 
    \[
      \depth (S/I(C_{n})) = \lceil (n-1)/3 \rceil. 
    \]
    Hence by Auslander--Buchsbaum Theorem, one has 
    \begin{displaymath}
      p := \projdim(S/I(C_{n})) = n - \depth(S/I(C_{n})) 
         = n - \lceil (n-1)/3 \rceil. 
    \end{displaymath}
    Also by \cite[Theorem 5.2]{BHT}, one has 
    \[
      r := \reg(S/I(C_{n})) = {
      \begin{cases} 
        \lfloor n/3 \rfloor, & \text{if $n \equiv 0, 1 \ {\rm mod} \ 3$}, \\ 
        \lfloor n/3 \rfloor + 1, & \text{if $n \equiv 2 \ {\rm mod} \ 3$}. 
      \end{cases}}
    \]

  \par
  Then we can prove the case where $n= 3 \ell$. 
  In this case, 
  $p = 2\ell$ and $r = \ell$. 
  For $1 \leq k \leq \ell$, let $B_{k}$ be the induced subgraph of $C_n$ 
  on $\{x_{3(k-1) + 1}, x_{3(k-1) + 2}, x_{3k}\}$. 
  Then $B_k$ is the star subgraph 
  $G^{{\rm star}(x_{3(k-1)+2})}_{2}$. 
  Take $e_{k} := \{ x_{3(k-1) + 1}, x_{3(k-1) + 2} \} \in E(B_{k})$. 
  Then $\{ e_{1}, \ldots, e_{\ell} \}$ forms an induced matching of $C_n$. 
  Thus Lemma \ref{nonvanishing} says that 
  $\beta_{p, p + r} (S/I(C_{n})) = \beta_{2\ell, 2\ell + \ell} (S/I(C_{n})) 
  \neq 0$. 
  Hence $S/I(C_{n})$ satisfies the equality $(*)$. 

  \par
  For the cases $n= 3 \ell +1, 3 \ell +2$, we compute all invariants 
  appearing in the equality $(\ast)$. 
  We have already known the dimension, the depth, and the regularity. 
  In order to compute $\deg h (S/I(C_n), \lambda)$, 
  consider the short exact sequence 
  \[
   0 \to S/I(C_{n}) : (x_{n}) (-1) \xrightarrow{\ \cdot x_{n} \ } S/I(C_{n}) \to S/I(C_{n}) + (x_{n}) \to 0. 
  \]
  Since $I(C_{n}) + (x_{n}) = (x_{n}) + I(P_{n-1})$, we have  
  \[
   S/I(C_{n}) + (x_{n}) \cong K[V(P_{n-1})]/I(P_{n-1}). 
  \]
  Also since $I(C_{n}) : (x_{n}) = (x_{1}, x_{n-1}) + (x_{2}x_{3}, 
  \ldots, x_{n-3}x_{n-2})$, we have  
  \begin{eqnarray*}
    S/I(C_{n}) \colon (x_{n}) &\cong& K[x_{2}, \ldots, x_{n-2}, x_{n}]/(x_{2}x_{3}, \ldots, x_{n-3}x_{n-2})\\ 
  &\cong& K[V(P_{n-3})]/I(P_{n-3}) \otimes_{K} K[x_{n}].  
  \end{eqnarray*}
  Thus Lemma \ref{HilbertSeries} says that 
  \begin{eqnarray*}
  H(S/I(C_{n}), \lambda) &=& H(S/I(C_{n}) + (x_{n}), \lambda) + \lambda H(S/I(C_{n}) \colon (x_{n}), \lambda) \\
  &=& \frac{h(K[V(P_{n-1})]/I(P_{n-1}), \lambda)}{(1-\lambda)^{\lceil (n-1)/2 \rceil}} 
  + \frac{\lambda h(K[V(P_{n-3})]/I(P_{n-3}), \lambda)}{(1-\lambda)^{\lceil (n-3)/2 \rceil + 1}} \\
  &=& \frac{h(K[V(P_{n-1})]/I(P_{n-1}), \lambda) + \lambda h(K[V(P_{n-3})]/I(P_{n-3}), \lambda)}{(1-\lambda)^{\lceil (n-1)/2 \rceil}}. 
  \end{eqnarray*}
%
  By (2), one has 
    \begin{displaymath}
      \begin{aligned}
        &\deg h(K[V(P_n)]/I(P_{n}), \lambda) \\
        &= \reg(K[V(P_n)]/I(P_{n})) + \dim K[V(P_n)]/I(P_{n}) 
           - \depth(K[V(P_n)]/I(P_{n})) \\
        &= \lceil (n-1)/3 \rceil + \lceil n/2 \rceil - \lceil n/3 \rceil \\
        &= {\begin{cases} 
             \lceil n/2 \rceil, 
             & \text{if $n \equiv 0, 2 \ {\rm mod} \ 3$}, \\ 
            \lceil n/2 \rceil - 1, 
            & \text{if $n \equiv 1 \ {\rm mod} \ 3$}. 
            \end{cases}}  \\
      \end{aligned}
    \end{displaymath}
  \begin{itemize}
  \item {\bf The case $n = 3\ell + 1$ : } 
    Then $\reg (S/I(C_{n})) = \depth (S/I(C_{n})) = \ell$ 
    and $\dim S/I(C_{n}) = \lceil 3\ell/2 \rceil$. 
    Moreover, since 
    \[
      \deg h(K[V(P_{n-1})]/I(P_{n-1}), \lambda) 
        = \deg h(K[V(P_{3\ell})]/I(P_{3\ell}), \lambda) 
        = \lceil 3\ell/2 \rceil
    \]
    and 
    \begin{displaymath}
      \begin{aligned}
        \deg h(K[V(P_{n-3})]/I(P_{n-3}), \lambda) 
          &= \deg h(K[V(P_{3\ell - 2})]/I(P_{3\ell - 2}), \lambda) \\
          &= \lceil (3\ell - 2)/2 \rceil - 1 = \lceil 3\ell/2 \rceil - 2, 
      \end{aligned}
    \end{displaymath}
    one has $\deg h(S/I(C_{n}), \lambda) = \lceil 3\ell/2 \rceil$. 
    Hence $S/I(C_{n})$ satisfies the equality $(*)$. 
  \item {\bf The case $n = 3\ell + 2$ : } 
    Then $\reg (S/I(C_{n})) = \depth (S/I(C_{n})) = \ell + 1$ 
    and $\dim S/I(C_{n}) = \lceil (3\ell + 1)/2 \rceil$. 
    Moreover, since 
    \[
      \begin{aligned}
      \deg h(K[V(P_{n-1})]/I(P_{n-1}), \lambda) 
        &= \deg h(K[V(P_{3\ell + 1})]/I(P_{3\ell + 1}), \lambda) \\
        &= \lceil (3\ell + 1)/2 \rceil - 1 
      \end{aligned}
    \]
    and 
    \[
      \begin{aligned}
        \deg h(K[V(P_{n-3})]/I(P_{n-3}), \lambda) 
          &= \deg h(K[V(P_{3\ell - 1})]/I(P_{3\ell - 1}), \lambda) \\
          &= \lceil (3\ell - 1)/2 \rceil = \lceil (3\ell + 1)/2 \rceil - 1, 
      \end{aligned}
    \]
    one has $\deg h(S/I(C_{n}), \lambda) = \lceil (3\ell + 1)/2 \rceil$. 
    Hence $S/I(C_{n})$ satisfies the equality $(*)$. 
  \end{itemize}
  \item
    Since $G_{s}$ has no cycle, 
    one has $\reg (S/I(G_{s})) = im (G) = 1$ by \cite{Z}. 
    Also it is easy to see from Lemma \ref{nonvanishing} that 
    $\projdim (S/I(G_{s})) = s+2$, 
    and $\beta_{s+2, (s+2)+1} (S/I(G_{s})) \neq 0$. \qed
  \end{enumerate}
\end{proof}

\begin{Remark}\normalfont
  The graph $G_s$ in Proposition \ref{other-ast} (as well as $P_{3\ell +1}$) 
  is an example of a graph 
  satisfying $(\ast)$ with $\deg h(S/I(G_{s}), \lambda) < \dim S/I(G_{s}) 
  (=s+2)$ 
  because $\reg (S/I(G_s)) = 1 < 2 = (s+4) - \projdim (S/I(G_s)) 
  = \depth (S/I(G_s))$. 
  Note that Cameron--Walker graphs $G$ satisfies  
  $\deg h(S/I(G), \lambda) = \dim S/I(G)$.
\end{Remark}


\section{Other properties on Cameron--Walker graphs}
\label{sec:application}
In this section, we provide some properties on a Cameron--Walker graph 
derived from the results of previous sections. 

\par
Let $G$ be a finite simple graph and $S = K[V(G)]$. 
Suppose that $S/I(G)$ is Cohen--Macaulay. 
Then the equalities $(*)$ and $\dim S/I(G) = \depth (S/I(G))$ hold. 
Hence one has $\deg h(S/I(G), \lambda) = \reg(S/I(G))$. 
Nevertheless, $\deg h(S/I(G), \lambda) = \reg(S/I(G))$ does not imply that 
$S/I(G)$ is Cohen--Macaulay, see \cite[Example 3.2]{HM1}. 
Moreover, in general, there is no relationship 
between the regularity and the degree of the $h$-polynomial. 
Actually, \cite{HMVT} proved that for given integers $r, s \geq 1$, 
there exists a finite simple graph $G$ such that 
$\reg(S/I(G)) = r$ and $\deg h(S/I(G), \lambda) = s$. 
%
However, we can derive from Proposition \ref{s=d} the relation between 
$\reg(S/I(G))$ and $\deg h(S/I(G), \lambda)$ when $G$ is Cameron--Walker. 
Moreover we provide a complete classification of Cameron--Walker graphs $G$ 
with $\deg h(S/I(G), \lambda) = \reg(S/I(G))$. 

\begin{Theorem}
  \label{s=r}
  Let $G$ be a Cameron--Walker graph whose labeling of vertices is 
  as in Figure \ref{fig:CameronWalkerGraph}. 
  Then we have $\deg h(S/I(G), \lambda) \geq \reg(S/I(G))$. 
  Moreover the equality $\deg h(S/I(G), \lambda) = \reg(S/I(G))$ 
  holds if and only if $s_{i} = 1$ for all $1 \leq i \leq m$ and 
  $t_{j} \geq 1$ for all $1 \leq j \leq n$. 
\end{Theorem}
\begin{proof}
  We first note that $\reg (S/I(G)) = \sum_{j = 1}^{n} t_{j} + m$.
  Combining this with Proposition \ref{s=d}, one has 
  \begin{displaymath}
    \deg h(S/I(G), \lambda) - \reg(S/I(G)) 
    = \left(\sum_{i = 1}^{m}s_{i} - m\right) 
    + \sum_{j = 1}^{n}\left( \max\left\{t_{j}, 1\right\} - t_{j} \right). 
  \end{displaymath}
  Note that each summands of right hand-side is non-negative. 
  Then the desired assertion follows. 
\end{proof}


\par
Let $G$ be a Cameron--Walker graph. 
Combining the inequality 
\begin{displaymath}
  \deg h\left(S/I(G), \lambda \right) - \reg\left( S/I(G) \right) 
  \leq \dim S/I(G) - \depth\left( S/I(G) \right)
\end{displaymath}
with Theorem \ref{CW}, Theorem \ref{s=r}, and Proposition \ref{depth2}, 
one has 
\begin{displaymath}
  \dim S/I(G) = \deg h\left(S/I(G), \lambda \right) \geq \reg (S/I(G)) \geq \depth (S/I(G)) \geq 2. 
\end{displaymath}
Then it is natural to ask the following 
\begin{Question}\normalfont
  \label{quest:dre}
  Given arbitrary integers $d, r, e$ with $d \geq r \geq e \geq 2$, 
  is there a Cameron--Walker graph $G$ 
  satisfying 
  \begin{displaymath}
    (\ast \ast) \quad 
      \dim S/I(G) = \deg h\left(S/I(G), \lambda \right) = d, \  
      \reg S/I(G) = r, \  
      \depth S/I(G) = e? 
  \end{displaymath}
\end{Question}

We have already investigated Cameron--Walker graphs $G$ 
with $\depth S/I(G) = 2$ in Proposition \ref{depth2}. 
Their invariants are as follows: 
  \begin{enumerate}
  \item[(e1)] $\dim S/I(G) = \deg h\left(S/I(G), \lambda \right) = s_1 + s_2 + n > 2 =  
    \reg (S/I(G)) = \depth (S/I(G))$. 
  \item[(e2)] $\dim S/I(G) = \deg h\left(S/I(G), \lambda \right) = s_1 + 1 \geq 2 
    = \reg (S/I(G)) = \depth (S/I(G))$.
  \item[(e3)] $\dim S/I(G) = \deg h\left(S/I(G), \lambda \right) = \reg (S/I(G)) = t_1 + 1 > 2 
    = \depth (S/I(G))$. 
  \end{enumerate}

Therefore we have the following answer for Question \ref{quest:dre} 
when $e = 2$. 
\begin{Corollary}
  Let $d, r, e$ be integers with $d \geq r \geq e= 2$. 
  Then there exists a Cameron--Walker graph $G$ satisfying 
    $(\ast \ast)$ if and only if $r=2$ or $r=d$. 
\end{Corollary}

When $e \geq 3$, we have the following answer for Question \ref{quest:dre}. 
\begin{Theorem}
  \label{dre>3}
  Given arbitrary integers $d, r, e$ with $d \geq r \geq e \geq 3$, 
  there exists a Cameron--Walker graph $G$ satisfying 
  $\dim S/I(G) = \deg h(S/I(G), \lambda) = d$, $\reg(S/I(G)) = r$, 
  and $\depth(S/I(G)) = e$. 
\end{Theorem}
\begin{proof}
  We use the labeling of vertices of a Cameron--Walker graph 
  as in Figure \ref{fig:CameronWalkerGraph}. 
  Set $V_{\rm bip} = \{ v_1, \ldots, v_m, \; w_1, \ldots, w_n \}$. 
  \newline
  {\bf $\bullet$ The case $d > r$:} 
  Let $G$ be the Cameron--Walker graph with $m=e-1$, $n=2$, 
  $s_1 = \cdots = s_{e-2} = 1$, $s_{e-1} = d-r$, $t_1 = r-e+1$, and 
  $t_2 =0$ such that 
  \begin{displaymath}
    E(G_{V_{\rm bip}}) = \big\{ \{ v_1, w_1 \}, 
      \; \{ v_1, w_2 \}, \{ v_2, w_2 \}, \ldots, \{ v_{e-1}, w_2 \} \big\}; 
  \end{displaymath}
  see Figure \ref{fig:d>r}. 
  \begin{figure}[htbp]
  \centering

  \bigskip

  \begin{xy}
    \ar@{} (0,0);(20, 0) *{\text{$G =$}};
    \ar@{} (0,0);(50, 6) *++!R{v_{1}} *\cir<4pt>{} = "A1";
    \ar@{-} "A1";(50, -6) *++!R{w_{1}} *\cir<4pt>{} = "A2";
    \ar@{} (0,0);(64, 6) *++!R{v_{2}} *\cir<4pt>{} = "B1";
    \ar@{-} "B1";(64, -6) *++!R{w_{2}} *\cir<4pt>{} = "B2";
    \ar@{-} "A1";"B2";
    \ar@{-} "B2";(90, 6) *++!L{v_{e-2}} *\cir<4pt>{} = "C";
    \ar@{-} "B2";(110, 6) *++!L{v_{e-1}} *\cir<4pt>{} = "D";
    \ar@{} (0,0); (77, 9) *++!U{\cdots};
    \ar@{-} "A1";(50, 18) *\cir<4pt>{};
    \ar@{-} "B1";(64, 18) *\cir<4pt>{};
    \ar@{-} "C";(90, 18) *\cir<4pt>{};
    \ar@{-} "D";(100, 18) *\cir<4pt>{};
    \ar@{-} "D";(120, 18) *\cir<4pt>{};
    \ar@{} (0,0); (110, 21) *++!U{\cdots};
    \ar@{-} "A2";(36, -18) *\cir<4pt>{} = "T1";
    \ar@{-} "A2";(40, -22) *\cir<4pt>{} = "T2";
    \ar@{-} "T1";"T2";
    \ar@{-} "A2";(64, -18) *\cir<4pt>{} = "T3";
    \ar@{-} "A2";(60, -22) *\cir<4pt>{} = "T4";
    \ar@{-} "T3";"T4";
    \ar@{} (0,0); (50, -12) *++!U{\cdots};
    \ar@{} (0,0);(110, 30) *{\text{$d - r$}};
    \ar@{} (0,0);(110, 24) *{\text{leaf edges}};
    \ar@{} (0,0);(50, -28) *{\text{$r - e + 1$}};
    \ar@{} (0,0);(50, -34) *{\text{pendant triangles}};
  \end{xy}

  \bigskip

  \caption{The Cameron--Walker graph $G$ in the proof of Theorem \ref{dre>3}}
  \label{fig:d>r} with $d>r$
  \end{figure}
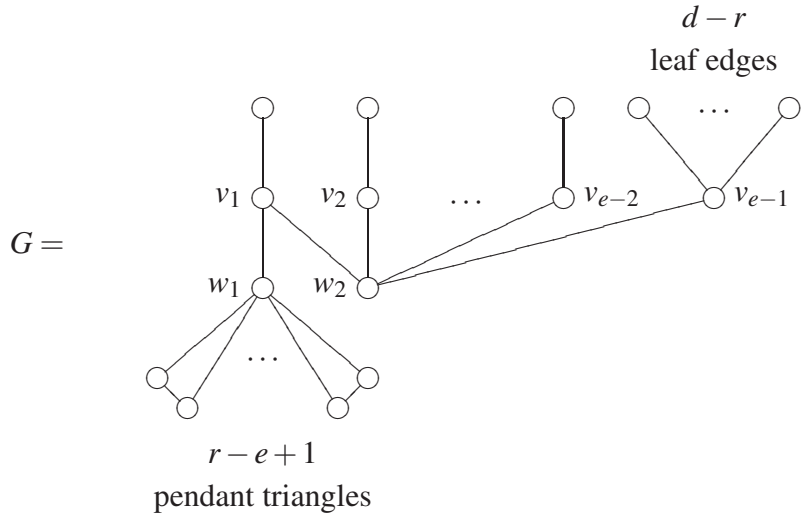
  Then it is easy to see that 
  $\dim (S/I(G)) = \deg h\left(S/I(G), \lambda \right) =d$ and $\reg (S/I(G)) = r$. 
  Also, $A := \{ v_2, \ldots, v_{e-1} \} \cup \{ x_1^{(1)}, w_1 \}$ 
  is an independent set of $V(G)$ with $A \cup N_G (A) = V(G)$ 
  which gives $i(G)$. Thus one has $\depth S/I(G) = i(G) = |A| = e$. 

  \par
  {\bf $\bullet$ The case $d = r$:} 
  Let $G$ be the Cameron--Walker graph with $m=e-1$, $n=1$, 
  $s_1 = \cdots = s_{e-1} = 1$, and $t_1 = d-e+1$. 
  Then it is easy to see that 
  $\dim (S/I(G)) = \deg h\left(S/I(G), \lambda \right) = \reg (S/I(G)) = d$. 
  Also $A := \{ x_1^{(1)}, \ldots, x_{e-1}^{(1)} \} \cup \{  w_1 \}$ 
  is an independent set of $V(G)$ with $A \cup N_G (A) = V(G)$ 
  which gives $i(G)$. Thus one has $\depth S/I(G) = i(G) = |A| = e$. 
\end{proof}

%



\bigskip

\noindent
{\bf Acknowledgment.}
The authors were partially supported by JSPS KAKENHI 26220701, 15K17507, 
17K14165 and 16J01549.

\end{document}